\documentclass{article}
\usepackage[utf8]{inputenc}
\usepackage{latexsym, amsfonts, amsmath, amssymb, amsthm, cite, hyperref}
\usepackage[margin=1in]{geometry}
\usepackage{comment}
\newtheorem{theorem}{Theorem}
\newtheorem{lemma}[theorem]{Lemma}
\newtheorem{corollary}[theorem]{Corollary}
\theoremstyle{definition}
\newtheorem{definition}[theorem]{Definition}
\newtheorem{example}[theorem]{Example}
\newtheorem*{unnumlemma}{Lemma}
\theoremstyle{remark}
\newtheorem{remark}[theorem]{Remark}

\usepackage{xcolor}

\newcommand{\nc}{\newcommand}
\newcommand{\renc}{\renewcommand}
\nc{\wt}{\widetilde}
\nc{\mc}{\mathcal}
\nc{\defeq}{:=}
\nc{\R}{\mathbb{R}}
\renc{\P}{\mathbb{P}}
\nc{\E}{\mathbb{E}}
\nc{\ol}{\overline}
\nc{\on}{\operatorname}
\nc{\Z}{\mathbb{Z}}
\nc{\Q}{\mathbb{Q}}
\nc{\be}{\begin{equation}}
\nc{\ee}{\end{equation}}

\title{\vspace*{-0.75in}Universality theorems for zeros of random real polynomials \\ with fixed coefficients}

\author{Matthew C. King and Ashvin A.~Swaminathan}

\date{\today}

\begin{document}

\maketitle

\begin{abstract}
\vspace*{-5pt}
Consider a monic polynomial of degree $n$ whose subleading coefficients are independent, identically distributed, nondegenerate random variables having zero mean, unit variance, and finite moments of all orders, and let $m \geq 0$ be a fixed integer. We prove that such a random monic polynomial has exactly $m$ real zeros with probability $n^{-3/4+o(1)}$ as $n\to \infty$ through integers of the same parity as $m$. More generally, we determine conditions under which a similar asymptotic formula describes the corresponding probability for families of random real polynomials with multiple fixed coefficients. Our work extends well-known universality results of Dembo, Poonen, Shao, and Zeitouni, who considered the family of real polynomials with all coefficients random. 

 As a number-theoretic consequence of these results, we deduce that an algebraic integer $\alpha$ of degree $n$ has exactly $m$ real Galois conjugates with probability $n^{-3/4+o(1)}$, when such $\alpha$ are ordered by the heights of their minimal polynomials.
\end{abstract}

\section{Introduction} \label{sec-theintro}

Let $(a_i)_{i \in \mathbb{N}}$  denote a sequence of independent and identically distributed (i.i.d.) random variables of zero mean and unit variance possessing finite moments of all order. Consider the polynomials
$$f_n(x) \defeq x^{n-1} + \sum_{i= 2}^n a_{n-i}x^{n-i}, \quad \text{and} \quad f_n^*(x) \defeq \sum_{i= 1}^n a_{n-i}x^{n-i}.$$
Note that $f_n$ is \emph{monic} with all subleading coefficients random, whereas $f_n^*$ has all coefficients random. For $n$ odd, let $P_n$ (resp., $P_n^*$) be the probability that $f_n(x)$ (resp., $f_n^*(x)$) is everywhere positive; i.e., define
$$P_n \defeq \mathbb{P}\big(f_n(x) > 0,\, \forall x \in \R\big), \quad \text{and} \quad P_n^* \defeq \mathbb{P}\big(f_n^*(x) > 0,\, \forall x \in \R\big).$$
In~\cite{MR1915821}, Dembo, Poonen, Shao, and Zeitouni (henceforth, DPSZ) studied the asymptotic behavior of $P_n^*$ in the limit as $n \to \infty$. Strikingly, they proved that there exists a universal positive constant $b$, independent of $n$, such that $P_n^* = n^{-b+o(1)}$ (i.e., the limit $\lim_{n \to \infty} \log_n P_{n}^*$ exists and equals $-b$); see Theorem~1.1(a) in \emph{loc.~cit}. They further showed that $0.4\leq b\leq 2$ and performed numerical simulations \mbox{indicating that $b \approx 0.76$.} It was later determined in works of Poplavskyi and Schehr~\cite[Equation (4)]{PhysRevLett.121.150601} and FitzGerald, Tribe, and Zaboronski~\cite[Proposition~5]{MR4499280} that $b = 3/4$; for further discussion of the literature on the constant $b$, see Remark~\ref{rmk-b} (to follow).

  In light of these remarkable results, it is natural to ask the following pair of questions:
\begin{enumerate}
\item Does the above asymptotic formula for the full family of random polynomials $f_n^*(x)$ admit an analogue for the special family of monic random 
polynomials $f_n(x)$? More specifically, is there a constant $b' > 0$ such that $P_n = n^{-b' + o(1)}$?
\item If the answer to the first question is yes, then how is $b'$ related to $b = 3/4$? Are they equal?
\end{enumerate}
In this paper, we give an affirmative answer to both of the above questions. Indeed, we prove:
\begin{theorem} \label{thm-mainmon}
    We have that $P_n = n^{-3/4 + o(1)}$. 
\end{theorem}

More generally, take $j \in \mathbb{N}$, take $n \geq j$ to be of parity different from that of $j$, and let $P_{n,j}$ (resp., $P_{n,j}^*$) be the probability that $f_n(x)$ (resp., $f_n^*(x)$) has exactly $j$ simple real zeros; i.e., define
$$P_{n,j} \defeq \P\big(\#\{x \in \R : f_n(x) = 0,\, f_n'(x) \neq 0\} = j\big), \,\, \text{and} \,\, P_{n,j}^* \defeq \P\big(\#\{x \in \R : f_n^*(x) = 0,\, f{_n^*}'(x) \neq 0\} = j\big).$$
In this setting, one can ask how $P_{n,j}$ and $P_{n,j}^*$ behave asymptotically as $n \to \infty$. When $j = 0$, we have that $P_{n,0} = P_n$ and $P_{n,0}^* = P_n^*$, and the aforementioned asymptotic formulas apply. In~\cite{MR1915821}, DPSZ prove that, for the full family of random polynomials $f_n(x)$, the same asymptotic formula holds --- i.e., we have that $P_{n,j}^* = n^{-3/4+o(1)}$ --- for each $j$, all the way up to $j = o(\log n/\log\log n)$; see Theorem~1.2 in \emph{loc.~cit}. Our next result gives the analogue for the family of random monic polynomials $f_n^*(x)$:

\begin{theorem} \label{thm-locmainmon}
    Let $j \in \mathbb{N}$ be such that $j \equiv n-1 \pmod 2$. Then we have that $P_{n,j} = n^{-3/4 + o(1)}$ as $n \to \infty$. In fact, $f_n(x)$ has at most $o(\log n/\log\log n)$ real zeros with probability $n^{-3/4 + o(1)}$.
\end{theorem}

Roughly speaking, our methods to prove Theorems~\ref{thm-mainmon} and~\ref{thm-locmainmon} involve separating the behavior of the leading term of the random monic polynomial from that of the remaining terms, which comprise a fully random polynomial of one smaller degree. The behavior of this lower-degree polynomial can then be analyzed using results obtained by DPSZ in the course of proving their asymptotic formulas in the context of fully random polynomials. It is natural to expect that similar arguments might work much more generally, to prove analogues of Theorems~\ref{thm-mainmon} and~\ref{thm-locmainmon} for random polynomials with multiple fixed coefficients. In this paper, we actually prove these more general theorems, the statements of which are given in \S\ref{sec-stategens}, and we deduce Theorems~\ref{thm-mainmon} and~\ref{thm-locmainmon} as consequences.

\subsection{Generalizations of Theorems~\ref{thm-mainmon} and~\ref{thm-locmainmon} to polynomials with multiple fixed coefficients} \label{sec-stategens}

As alluded to above, our methods allow us to generalize Theorems~\ref{thm-mainmon} and~\ref{thm-locmainmon} in two different directions. Firstly, we can fix the values of several coefficients, not just the leading coefficient; and secondly, we can replace the condition of positivity with a much stronger condition. To state these generalizations, we require some further notation. Let $k \in \Z_{> 0}$ be fixed, let $S \subset \{1, \dots, k\}$ be a subset containing $k$, and for each $i \in S$, fix a number $c_i \in \R$; if $1 \in S$, take $c_1 > 0$. Consider the polynomial
$$f_{n,S}(x) \defeq \sum_{i \in S} c_ix^{n-i} + \sum_{\substack{i \in \{1, \dots,n\} \\ i \not\in S}} a_{n-i}x^{n-i},\vspace*{-0.2cm}$$
which has random coefficients with the exception of the terms $c_i x^{n-i}$ for $i \in S$. Let $(\gamma_n)_n$ be a sequence of non-random functions on $\R$, and assume that exists $\delta > 0$ for which $n^{\delta}|\gamma_n(x)| \to 0$ uniformly over $x \in \R$. For $n$ odd, let $P_{n,S,\gamma_n}$ be the probability that the \emph{normalized} polynomial
\begin{equation} \label{eq-thehats}
\hat{f}_{n,S}(x) \defeq \frac{f_{n,S}(x)}{\sqrt{\E\big(f_{n,S}(x)^2\big)}}
\end{equation}
is everywhere bigger than $\gamma_n(x)$; i.e., define
$$P_{n,S,\gamma_n} \defeq \mathbb{P}\big(\hat{f}_{n,S}(x) > \gamma_n(x),\, \forall x \in \R\big).$$
Note that demanding $\hat{f}_{n,S}(x)$ to be everywhere greater than $\gamma_n(x)$ is a considerably stronger condition than merely asking $f_{n,S}(x)$ to be everywhere positive!

One cannot reasonably expect that an unconditional analogue of Theorem~\ref{thm-mainmon} would hold with $P_n$ replaced by, say, $P_{n,S,0}$ (where by the subscript ``$0$'' we mean that $\gamma_n  \equiv 0$ is taken to be identically zero for each $n$). Indeed, one can choose the \emph{coefficient data} --- i.e., the data of the set $S$, the fixed coefficients $c_i$ for $i \in S$, and the distribution of the random coefficients $a_i$ --- in such a way that $P_{n,S,0} = 0$ for infinitely many $n$. As we detail in Example~\ref{eg-probzero} (to follow), one way to do this is to choose the distribution of the $a_i$ to be bounded in absolute value by some constant $C$, and to choose some of the $c_i$ to be considerably less than $-C$. We then analyze the positivity of $f_{n,S}(x)$ on three ranges of $|x|$. First, when $|x|< 1$, the fixed terms matter little, and the lower-degree random terms dominate, making $f_{n,S}(x) > 0$ with positive probability. The same  occurs when $|x| -1$ is sufficiently large, in which case the leading term is positive and dominates over all other terms. But in the middle range, when $0 \leq |x| -1$ is sufficiently small, the fixed terms can contribute large negative quantities to the value of $f_{n,S}(x)$, forcing it to be negative regardless of the values of the random coefficients. 

The upshot is that, if we are to prove that $P_{n,S,\gamma_n}$ obeys an asymptotic similar to that which we obtained for $P_n$ in Theorem~\ref{thm-mainmon}, then we must impose a nontrivial condition on the coefficient data, one that addresses the potential for $f_{n,S}(x)$ to be negative on the aforementioned ``middle range'' of $|x|$. To this end, we shall stipulate the coefficient data be ``nice,'' where the notion of niceness is defined as follows: 

\begin{definition} \label{def-nice}
    With notation as above, we say that the coefficient data are \emph{nice} if there exists an even integer $s > k$ such that we have 
    $$\mathbb{P}\big(f_{s,S}(x) \neq 0, \, \forall x \text{ s.t.~} |x| > 1\big) > 0.$$
\end{definition}

Conditional on the coefficient data being nice, we prove the following generalization of Theorem~\ref{thm-mainmon} about the asymptotic behavior of $P_{n,S,\gamma_n}$:
\begin{theorem} \label{thm-main}
  We have the upper bound $P_{n,S,\gamma_n} \leq n^{-3/4 + o(1)}$. Furthermore, if the chosen data are nice, then we have equality $P_{n,S,\gamma_n} = n^{-3/4 + o(1)}$.
\end{theorem}

\begin{remark}
     We note that many natural choices of coefficient data are nice. A few interesting examples of such choices are listed as follows:
     \begin{enumerate}
     \item Take the $a_i$ to be standard normal random variables, with any choice of fixed coefficients. In fact, the main results of~\cite{MR1915821}, along with Theorems~\ref{thm-main} and~\ref{thm-locmaingen} in the present paper, are proven by using strong approximation results to reduce to the Gaussian case; see Theorem~\ref{thm-maingauss} below.
     \item Take the distribution of the $a_i$ to be arbitrary (of zero mean, unit variance, and finite moments of all orders), but take $S$ to consist entirely of odd numbers and take $c_i > 0$ for each $i \in S$, so that the fixed terms have even degree and positive coefficients. This includes the case of monic polynomials considered in Theorems~\ref{thm-mainmon} and~\ref{thm-locmainmon}. In particular, these theorems follow immediately from Theorem~\ref{thm-main} and~\ref{thm-locmaingen} by setting $S = \{1\}$, $c_1 = 1$, and $\gamma_n \equiv 0$ for all $n$.
     \item Take the distribution of the $a_i$ to have sufficiently large support, relative to the data of the fixed coefficients. More precisely, for every choice of $S$ and $(c_i)_{i \in S}$, there exists a constant $M > 0$ depending on these choices such that the coefficient data are nice if $\P(a_i > M) > 0$.
     \end{enumerate}
     \end{remark}

As mentioned above, the key to proving Theorem~\ref{thm-main} is to prove the corresponding result where the random coefficients $a_i$ are assumed to be Gaussians. In this case, the result of Theorem~\ref{thm-main} holds under a much weaker assumption than stipulating that the functions $\gamma_n$ converge everywhere to zero. Indeed, we prove:

    \begin{theorem} \label{thm-maingauss}
        Suppose that $(a_i)_i$ are standard normal random variables, that $\sup\{\gamma_n(x) : x \in \R,\, n \in \mathbb{N}\} < 1$, and that there exists a sequence $(\varepsilon_n)_n$ of positive real numbers with limit $0$ such that
        $$\sup\big\{|\gamma_n(x)| : \big||x|-1\big| \leq n^{-\varepsilon_n}\big\} \to 0.$$
        as $n \to \infty$. Then we have that $P_{n,S,\gamma_n} = n^{-3/4+o(1)}$.
    \end{theorem}

    We also prove the following generalization of Theorem~\ref{thm-locmainmon} concerning the probability $P_{n,j,S}$ that $f_{n,S}(x)$ has exactly $j$ zeros; once again, the niceness condition is required for the lower bound:

\begin{theorem} \label{thm-locmaingen}
    Let $j\in \mathbb{N}$ be such that $j \equiv n-1 \pmod 2$. Then we have the upper bound $P_{n,j,S} \leq n^{-3/4 + o(1)}$. If the chosen data are nice, then we have equality $P_{n,j,S} = n^{-3/4 + o(1)}$. Furthermore, the polynomial $f_n(x)$ has at most $o(\log n/\log\log n)$ real zeros with probability at most $n^{-3/4 + o(1)}$, with equality if the \mbox{coefficient data are nice.}
\end{theorem}
It is important to note that our proof of Theorem~\ref{thm-locmaingen} does \emph{not} rely on Theorem~\ref{thm-main} or its proof. Observe that if we take $n$ to be odd and $j$ to be even in Theorem~\ref{thm-locmaingen}, we obtain Theorem~\ref{thm-main} in the special case where $\gamma_n \equiv 0$. Further taking $S = \{1\}$ and $c_1 = 1$ gives a second proof of Theorem~\ref{thm-mainmon}. On the other hand, the much stronger Theorem~\ref{thm-main} cannot be similarly deduced from Theorem~\ref{thm-locmaingen}.

 Incidentally, our methods allow us to obtain some results in cases where the number of fixed coefficients is allowed to grow slowly with $n$. A particular case of interest is that of monic polynomials with many consecutive coefficients fixed to be zero. To set this up, let $k(n)$ be a nondecreasing function with values in $\mathbb{N}$, let $S_n = \{1, \dots, k(n)\}$, let $c_1 = 1$ and $c_i = 0$ for all $i \geq 2$, and consider the random polynomial $f_{n,S_n}(x)$. Explicitly, we have
      $$f_{n,S_n}(x) = x^{n-1} + \sum_{i = k(n)+1}^n a_{n-i}x^{n-i}.$$
      Then we have the following analogue of Theorem~\ref{thm-main}, which holds in the regime where $\log k(n)$ grows slower than $\sqrt{\log n}$ (with a somewhat weaker condition in the Gaussian case):
    \begin{theorem} \label{thm-fixedcoeffszero}
    Let notation be as above, and suppose that $k(n) = n^{o(1)}$ in the case of Gaussian coefficients, and that $k(n) = n^{o(1/\sqrt{\log n})}$ in the case of general coefficients. Then we have
    $P_{n, S_n,\gamma_n} = n^{-3/4 + o(1)}$.
    \end{theorem}
    The following analogue of Theorem~\ref{thm-locmaingen} holds in the tighter regime where $k(n)$ grows slower than $\log n$.
    \begin{theorem} \label{thm-fixedcoeffszero2}
    Let notation be as above, and suppose that $k(n) = o(\log n)$. Let $j\in \mathbb{N}$ be such that $j \equiv n-1 \pmod 2$.  Then we have
    $P_{n, j,S_n} = n^{-3/4 + o(1)}$.
    \end{theorem}

\subsection{Application to counting algebraic integers}

As a consequence of Theorem~\ref{thm-locmaingen}, we can deduce asymptotic counts of algebraic integers having few or no real Galois conjugates. Before stating this application, we must define how we count algebraic \mbox{integers --- i.e.,} we must put a height function on them. Given an algebraic integer $\alpha$, let $\alpha'$ be the unique $\Z$-translate of $\alpha$ with trace lying in $\{0,\dots, n-1\}$, and let $p(x) = x^n + \sum_{i = 1}^n p_{i}x^{n-i} \in \Z[x]$ be the minimal polynomial of $\alpha'$. Then we define the \emph{height} of the equivalence class of $\alpha$ to be $\max_{i \in \{1,  \dots, n\}} |p_i|^{1/i}$.

We say that a nonzero algebraic integer $\alpha$ is \emph{$j$-realizable} if it has exactly $j$ real Galois conjugates. For instance, an algebraic integer $\alpha$ of degree $n$ is $0$-realizable if it is totally complex and $n$-realizable if it is totally real. Then Theorem~\ref{thm-locmaingen} has the following immediate corollary giving an asymptotic formula for the density of $j$-realizable algebraic integers:

\begin{corollary} \label{cor-algint}
    Let $j \in \mathbb{N}$ be fixed. When algebraic integers $\alpha$ of degree $n \equiv j \pmod 2$ are ordered by height, the density of $\alpha$ that are $j$-realizable is $n^{-3/4+o(1)}$.
\end{corollary}

Indeed, note that Corollary~\ref{cor-algint} follows from Theorem~\ref{thm-locmaingen} by taking the distribution of the $a_i$ to be the uniform distribution on $[-\sqrt{3},\sqrt{3}]$, and taking $k = 1$, $S = \{1\}$, and $c_1 = \sqrt{3}$. Note that the factor of $\sqrt{3}$ arises because the $a_i$ are required to have unit variance.

\begin{remark}
The problem of counting algebraic integers satisfying interesting conditions at the archimedean places has been studied before in the literature. Of particular relevance to the present paper is the work of Calegari and Huang~\cite{MR3687945} (cf.~the closely related earlier work of Akiyama and Peth\H{o}~\cite{MR3238322,MR3237072}), who consider algebraic integers ordered by a somewhat different height function, namely, the absolute value of the largest Galois conjugate. With respect to this height, they determine precise asymptotics for densities of various types of algebraic integers, including the totally complex subfamily, for which they obtain a density $\asymp n^{-3/8}$. Proving this result amounts to studying monic real polynomials, all of whose roots have absolute value at most $1$; note that such polynomials are special in that they are positive on $\R$ if and only if they are positive on $[-1,1]$. We note that, by modifying the proof of Theorem~\ref{thm-mainmon}, it is possible to show that a random monic polynomial is positive on $[-1,1]$ with probability $n^{-3/8 + o(1)}$. 
\end{remark}

\begin{remark}
     Our main results have other applications of arithmetic interest. For instance, by combining this theorem with the main results of~\cite{MR4703126}, one obtains an tighter upper bound on the proportion of superelliptic equations having no integral solutions. Specifically, in~\cite{MR4703126}, the proportion of  superelliptic equations having no integral solutions is bounded in terms of various local densities, one of which happens to be the density of \emph{monic} polynomials of bounded height with a specified number of real roots. To control this density, the trivial bound of $1$ is used; however, applying Theorem~\ref{thm-locmainmon} instead of the trivial bound would tighten the main results of~\cite{MR4703126} by a factor of $o(1)$. This is in complete analogy with how the results of~\cite{MR1915821} were applied in~\cite{thesource} and in~\cite{MR3600041} to obtain bounds on the proportion of hyperelliptic curves having no rational points over odd-degree number fields. Here, the proportion of such ``pointless'' curves is bounded in terms of various local densities, one of which happens to be the density of \emph{not-necessarily-monic} polynomials of bounded height with a specified number of real roots (hence the relevance of~\cite{MR1915821}).
\end{remark}

\subsection{Summary of related earlier work}
For at least the past century, mathematicians have been interested in the number of real zeros, say $N_n$ of random degree $n - 1$ polynomials of the form 
$\sum_{i = 1}^n a_{n - i} A_{n - i} x^{n - i}$ where the $A_{n - i}$ are deterministic real numbers depending on $n$ and $i$, and the $a_{n - i}$ are i.i.d.~with zero mean and unit variance, as above. The case of Kac polynomials, which arise from setting each $A_{n - i} = 1$ has been of particular interest. In 1932, Bloch and Polya considered $a_i$ chosen uniformly at random from $\{-1, 0, 1\}$ and showed $\mathbb{E}(N_n) = O(n^{1/2})$ \cite{blochpolya}. 
Beginning in 1938, Littlewood and Offord followed this with a series of papers \cite{LO1, LO2, LO3} studying $N_n$ for three coefficient distributions: standard normal, uniform on $[-1, 1]$, and Bernoulli on $\{-1, 1\}$. In each case, they showed that there exist constants $A, B > 0$ such that the following holds:
\[
A \frac{\log n}{(\log \log n)^2} \le N_n \le B (\log n)^2 
\qquad \text{w.p.~$1 - o(1)$ as $n \to \infty$.}
\]
In 1943, Kac found an exact formula for the density function of $N_n$ for any coefficient distribution with zero mean and finite variance, which he used to obtain the estimate $\mathbb{E}(N_n) = (\frac{2}{\pi} + o(1)) \log n$ in the case of the standard Gaussian distribution \cite{kac1}. Kac stated that along with the central limit theorem, his work would give the same formula for $\mathbb{E}(N_n)$ for many other coefficient distributions. 

Kac's work was the first suggestion of the universality phenomenon in the context of $N_n$. In probability theory, \textit{universality} refers to the case when a limiting random object as $n$ goes to infinity (such as the real number $\lim_{n \to \infty} \frac{\mathbb{E}(N_n)}{\log n}$) is insensitive to the particular distributions of i.i.d.~atoms (such as coefficients $a_i$) from which the random object is obtained. The study of universality for the quantity $N_n$ has attracted significant interest from probabilists and number theorists.  
Indeed, it is not only the expectation $\lim_{n \to \infty} \frac{\mathbb{E}(N_n)}{\log n}$ that exhibits universality. 
In several papers \cite{IM1, IM2, IM3, Maslova} ending in Maslova's work in 1974, Ibragimov and Maslova further illuminated the distribution of $N_n$, culminating in a central limit theorem for $N_n$. 
In 2002, DPSZ demonstrated with their study of $\mathbb{P}(N_n = 0)$ that it was not just the limiting shape of the CDF in Maslova's work but also the extreme left tail of the CDF for $N_n$ that exhibits a universal behavior, independent of the law of $a_i$. When some of the leading coefficients are fixed, our Theorem \ref{thm-main} reveals a clean dichotomy: unless the choices of fixed coefficients and random coefficient distribution force the problem to be degenerate in a precise way, the limiting behavior of $\mathbb{P}(N_n = 0)$ is universal with respect to the random coefficient distribution and the fixed coefficients. 

DPSZ state that their interest in $P_n^*$ arose from work of Poonen and Stoll in arithmetic geometry, which found the probability that a random hyperelliptic curve $y^2 = f(x)$ of genus $g$ over $\Q$ has odd Jacobian \cite{poonenstoll}. Poonen and Stoll reduced the problem to the computation of local probabilities, one for each completion of $\Q$, and the archimedean completion required the computation of the probability that $y^2 = f(x)$ has no real points, or, equivalently, that $f(x) < 0$ holds on all of $\R$. Importantly, this probability arose with the coefficients of $f$ having a uniform distribution. Further, for the application of DPSZ's work in \cite{thesource} and \cite{MR3600041} and for the application of our main results to \cite{MR4703126} and our Corollary 10, the relevant case is again that of uniformly distributed $a_i$. Meanwhile, the work of DPSZ shows that the probability $P_n^*$ is easiest to find when the coefficient distribution is Gaussian. Indeed, for DPSZ, handling the uniform distribution is no easier than handling a general distribution with zero mean, unit variance, and finite moments, and we find the same in our work. Even for the problem of determining asymptotics of $\mathbb{E}(N_n)$, Kac found the Gaussian distribution most tractable and did not solve the uniform case until six years later \cite{kac2}. Thus, when studying the real zeros of $f_{n, S}(x)$, proving universality reduces the arithmetically interesting case of the uniform distribution to the case of the Gaussian distribution, \mbox{most amenable to techniques from probability theory.}

\begin{remark} \label{rmk-b}
    The problem of determining both the physical meaning and the precise value of the constant $b$ introduced by DPSZ has garnered significant attention in the literature. Years after DPSZ, Schehr and Majumdar~\cite{PhysRevLett.99.060603,MR2415102} proved that $b$ is the so-called \emph{persistence exponent} of the two-dimensional heat equation. Other works, including those of Li and Shao~\cite{MR1902188} and Molchan~\cite{MR2999109} improved on the bounds on $b$ obtained by DPSZ. Meanwhile, in the same direction as some of the work of Schehr and Majumdar, Dembo and Mukherjee generalized the results of DPSZ to the case of independent coefficients with polynomially growing variance~\cite{MR3298469}. Despite studying a different family of random polynomials, Calegari and Huang were motivated, after their aforementioned result, to ask whether $b = 3/4$~\cite{MR3687945}.
    Finally, almost two decades after DPSZ,  Poplavskyi and Schehr~\cite[Equation (4)]{PhysRevLett.121.150601} used the truncated real ensemble of orthogonal matrices to show that $b = 3/4$. More recently, FitzGerald, Tribe, and Zaboronski~\cite[Proposition~5]{MR4499280} gave a different derivation of the value of $b$, using certain asymptotic results for Fredholm Pfaffians. For the sake of clarity, we choose to write $b$ for the constant in the rest of the paper, even though its value is known to be $3/4$. 
\end{remark}

\begin{remark}
The question of how often a polynomial over $\R$ has a specified number of zeros admits an interesting analogue over the nonarchimedean completions of $\Q$. Specifically, for a prime number $p$, one can ask: what is the probability that a random polynomial (or random monic polynomial) with $p$-adic integer coefficients has exactly $m$ roots in the fraction field $\Q_p$? This and other related questions were studied in detail by Bhargava, Cremona, Fisher, and Gajovi\'{c}~\cite{MR4422614}, who proved that the desired probability is a rational function of $n$, $m$, and $p$. Intriguingly, they show that this rational function is invariant under $p \mapsto 1/p$; this invariance phenomenon was subsequently demonstrated by G., Wei, and Yin to be a consequence of Poincar\'{e} duality for the zeta-functions of certain relevant varieties
\cite{g2023chebotarev}.
\end{remark}

\subsection{Organization}

The rest of this paper is organized as follows. We start in \S\ref{sec-gaussproofs} by proving Theorem~\ref{thm-maingauss}, and then in \S\ref{sec-pfgencoffs}, we deduce Theorems~\ref{thm-main} and~\ref{thm-fixedcoeffszero} (and hence also Theorem~\ref{thm-mainmon}) from Theorem~\ref{thm-maingauss}. We finish in \S\ref{sec-thelast} by proving Theorems~\ref{thm-locmaingen} and~\ref{thm-fixedcoeffszero2} (and hence also Theorem~\ref{thm-locmainmon} and Corollary~\ref{cor-algint}). Each of \S\S\ref{sec-gaussproofs}--\ref{sec-thelast} is divided into two subsections, the first of which proves the relevant ``lower bound'' --- i.e., that the desired probability is at least $n^{-b + o(1)}$ --- and the second of which proves the corresponding ``upper bound'' --- i.e., that the desired probability \mbox{is at most $n^{-b + o(1)}$.}

\section{Proof of Theorem~\ref{thm-maingauss} --- the case of Gaussian coefficients} \label{sec-gaussproofs}

In this section, we prove Theorem~\ref{thm-maingauss}, which is a version of our main result where the random coefficients $(a_i)_i$ are taken to be standard normal random variables. Where necessary, we highlight how the proof readily adapts to obtain Theorem~\ref{thm-fixedcoeffszero} \mbox{in the Gaussian case.}

In the rest of this section, we take $k$, $S$, $(c_i)_{i \in S}$, and $(\gamma_n)_n$ to be as in the setting of Theorem~\ref{thm-maingauss}. Following notation used in~\cite{MR1915821} for the Gaussian case, we denote the random degree-$n$ polynomial with fixed coefficients as $f_{n,S}^b(x)$ and the fully random degree-$n$ polynomial as $f_n^{b,*}(x)$, and we denote the corresponding normalized polynomials as $\hat{f}_{n,S}^b(x)$ and $\hat{f}_n^{b,*}(x)$, respectively. The proof makes crucial use of Slepian's lemma, which is a Gaussian comparison inequality and can be stated as follows (c.f.~\cite[p.~49]{adler1990introduction}):

\begin{unnumlemma}[Slepian] Let $X_t$ and $Y_t$ be two centered Gaussian processes of equal variance on a subset $T \subset \R$. (I.e., take $X_t$ and $Y_t$ such that $\E\big(X_t\big) = \E\big(Y_t\big) = 0$ and $\E\big(X_t^2\big) = \E\big(Y_t^2\big)$ for all $t \in T$.) Suppose that $X_t$ has everywhere greater covariance, so that for all $s, t \in T$, we have $\operatorname{Cov}\big(X_s, X_t\big) \ge \operatorname{Cov}\big(Y_s, Y_t\big)$. Then, for any $\lambda \in \R$, we have $\P\big(\inf_{t \in T} X_t \ge \lambda\big) \ge \P \big(\inf_{t \in T} Y_t \ge \lambda\big)$. 
\end{unnumlemma}

\subsection{Lower bound} \label{sec-low}

To obtain the lower bound, we split into cases according to the parity of $k$. In \S\ref{sec-lowkeven}, we handle the case where $k$ is even, and in \S\ref{sec-lowkodd}, we explain how the argument of \S\ref{sec-lowkeven} can be modified when $k$ is odd. 

\subsubsection{The case where $k$ is even} \label{sec-lowkeven}

Let $k$ be even. By moving the terms $a_{n-1}x^{n-1},\, \dots,\, a_{n-k}x^{n-k}$ to the right-hand side and renormalizing, one readily verifies that the condition $\hat{f}_n^b(x) > \gamma_n(x)$ is equivalent to the following condition:
\begin{equation} \label{eq-equivI1}
    \hat{f}_{n-k}^{b,*}(x) >  \gamma_n(x)\sqrt{1+\frac{ \underset{{i \in  S}}{\sum} c_{i}^2x^{2(n-i)} + \underset{i \in \{1, \dots, k\} \smallsetminus S}{\sum} x^{2(n-i)}}{ \underset{i=1}{\overset{n-k}{\sum}} x^{2(n-k-i)}}} -\frac{\underset{i \in S}{\sum} c_ix^{n-i} + \underset{i \in \{1,\dots,k\} \smallsetminus S}{\sum} a_{n-i}x^{n-i}}{\sqrt{\underset{i = 1}{\overset{n-k}{\sum}}x^{2(n-k-i)}}}. 
\end{equation}
For the purpose of obtaining a lower bound, we can restrict $a_{n-i}$ to lie within the interval $[1/2,1]$ for each $i \in \{1, \dots, k\} \smallsetminus S$. After making this restriction, we can also strengthen the condition~\eqref{eq-equivI1} (i.e., make it less probable) by replacing the right-hand side with a larger quantity. To this end, for each $i \in \{1,\dots,k\} \smallsetminus S$, let $\chi_i \colon \R \to \{0,1/2,1\}$ be the function defined as follows:
$$\chi_i(x) = \begin{cases} 1, & \text{ if $i$ is even and $x < 0$,} \\
1/2, & \text{ if $i$ is odd,} \\ 
0, & \text{ otherwise.}\end{cases}$$
Then define $\gamma_n^*(x)$ by
\begin{equation} \label{eq-gn*}
\gamma_n^*(x) \defeq \gamma_n(x)\sqrt{1+\frac{ \underset{{i \in  S}}{\sum} c_{i}^2x^{2(n-i)} + \underset{i \in \{1, \dots, k\} \smallsetminus S}{\sum} x^{2(n-i)}}{ \underset{i=1}{\overset{n-k}{\sum}} x^{2(n-k-i)}}}  -\frac{\underset{i \in S}{\sum} c_ix^{n-i} + \underset{i \in \{1,\dots,k\} \smallsetminus S}{\sum} \chi_i(x)x^{n-i}}{\sqrt{\underset{i = 1}{\overset{n-k}{\sum}}x^{2(n-k-i)}}}.
\end{equation}
Then the condition~\eqref{eq-equivI1} is implied by the stronger condition $\hat{f}_{n-k}^{b,*}(x) > \gamma_n^*(x)$, and so we have that
\begin{align}
    \P\big(\hat{f}_{n,S}^b(x) > \gamma_n(x),\, \forall x \in \R\big) & \geq \P\big(\hat{f}_{n,S}^b(x) > \gamma_n(x),\, \forall x \in \R; \text{ and } a_{n-i} \in [1/2,1],\, \forall i \in \{1, \dots, k\} \smallsetminus S\big) \nonumber \\
    & \geq \P\big(\hat{f}_{n-k}^{b,*}(x) > \gamma_{n}^*(x),\, \forall x \in \R; \text{ and } a_{n-i}\in [1/2,1],\, \forall i \in \{1, \dots, k\} \smallsetminus S\big) \nonumber \\
    & = \P\big(\hat{f}_{n-k}^{b,*}(x) > \gamma_{n}^*(x),\, \forall x \in \R\big)\times \prod_{\substack{i \in \{1,\dots, k\} \\ i \not\in S}} \P(a_{n-i} \in [1/2,1]). \label{eq-indepprod}
\end{align}
In~\eqref{eq-indepprod}, the product over $i$ is a nonzero constant, depending only on $k$ and $S$. Note that this product is trivial in the setting of Theorem~\ref{thm-fixedcoeffszero}. Thus, it suffices to bound the first probability in~\eqref{eq-indepprod}.

Our assumptions on $\gamma_n(x)$ in the statement of Theorem~\ref{thm-maingauss} imply that there exist integers $\tau_n$ such that $\log\log\log n \lll \tau_n \lll \log n$ and such that $\sup\{\gamma_n(x) : |x| \in [1 - \xi_n,(1-\xi_n)^{-1}]\} \to 0$, where $\xi_n = e^{-\tau_n}$.\footnote{Here, we say that $a \lll a'$ if $a = o(a')$ as $n \to \infty$.} In the setting of Theorem~\ref{thm-fixedcoeffszero}, this changes slightly: we take $\tau_n \ggg \log k(n)$, instead of $\log\log\log n$. Since $k$ is even, the covariance of $f_{n-k}^{b,*}(x)$ is everywhere nonnegative,\footnote{Note that this is not true in the case where $k$ is odd, which therefore necessitates separate treatment (see \S\ref{sec-lowkodd}).} which allows us to apply Slepian's lemma. Doing so, we find that
\begin{align} 
    & \P\big(\hat{f}_{n-k}^{b,*}(x) > \gamma_{n}^*(x),\, \forall x \in \R\big) \geq \label{eq-sleptosplit} \\
    & \qquad \P\big(\hat{f}_{n-k}^{b,*}(x) > \gamma_{n}^*(x),\, \forall x \in [0,(1-\xi_n)^{-1}]\big) \times \P\big(\hat{f}_{n-k}^{b,*}(x) > \gamma_{n}^*(x),\, \forall x \in [-(1-\xi_n)^{-1},0]\big)\times \nonumber \\
    & \qquad\qquad \P\big(\hat{f}_{n-k}^{b,*}(x) > \gamma_{n}^*(x),\, \forall x \in [(1 - \xi_n)^{-1},\infty)\big)  \times \P\big(\hat{f}_{n-k}^{b,*}(x) > \gamma_{n}^*(x),\, \forall x \in (-\infty,-(1 - \xi_n)^{-1}]\big).
    \nonumber
\end{align}
We start by handling the first two probabilities on the right-hand side of~\eqref{eq-sleptosplit}. The idea is to verify that these probabilities are of the form studied in~\cite[\S3]{MR1915821}. Indeed, it follows immediately from the results of~\cite[\S3]{MR1915821} in conjunction with Slepian's lemma that each of the first two probabilities on the right-hand side of~\eqref{eq-sleptosplit} is at least $n^{-b/2+o(1)}$, as long as the following two conditions are satisfied:
\begin{align}
& \sup\big\{\gamma_n^*(x) : x \in [-1+\xi_n,1-\xi_n],\, n \in \mathbb{N}\big\} < \infty, \quad \text{and} \label{eq-whatsneedlow-1} \\
& \sup\big\{\big|\gamma_{n}^*(x)\big| : |x| \in [1-\xi_n,(1-\xi_n)^{-1}]\big\} \to 0  \label{eq-whatsneedlow1}
\end{align}
as $n \to \infty$. To verify the conditions~\eqref{eq-whatsneedlow-1} and~\eqref{eq-whatsneedlow1}, note that we have the following useful bound:
\begin{equation} \label{eq-boundsecterm}
\sup\left\{\frac{|x^{n-j}|}{\sqrt{\sum_{i = 1}^{n-k} x^{2(n-k-i)}}}: |x| \in [0,(1-\xi_n)^{-1}]\right\} = \frac{(1-\xi_n)^{-(n-j)}}{\sqrt{\sum_{i = 1}^{n-k}(1-\xi_n)^{-2(n-i)}}}\ll \frac{\sqrt{\xi_n}}{(1-\xi_n)^{k-j}},
\end{equation}
where the implied constant is absolute. In the first step of~\eqref{eq-boundsecterm}, we have used the fact that the quantity being maximized is increasing in $|x|$. Summing the bound~\eqref{eq-boundsecterm} over $j \in \{1, \dots, k\}$ using our assumption that $\log k \lll \tau_n$, which implies that $k \lll 1/\sqrt{\xi_n} = e^{\tau_n/2}$ (note that this holds in both the settings of Theorems~\ref{thm-maingauss} and~\ref{thm-fixedcoeffszero}), we find that
\begin{equation} \label{eq-bipartite}
\sum_{j = 1}^k \frac{\sqrt{\xi_n}}{(1-\xi_n)^{k-j}} \ll \sqrt{\xi_n} \times \frac{(1 - \xi_n)^{-k}-1}{(1 - \xi_n)^{-1}-1} \ll \sqrt{\xi_n} \times \frac{k\xi_n}{\xi_n}  \ll k\sqrt{\xi_n},
\end{equation}
which tends to zero as $n \to \infty$. From this, we deduce a few consequences: first, the factor multiplied by $\gamma_n(x)$ in~\eqref{eq-gn*} is bounded for $\lvert x \rvert \leq (1 - \xi_n)^{-1}$; and second, the terms being subtracted on the right-hand side of~\eqref{eq-gn*} converge to zero uniformly over $|x| \leq (1-\xi_n)^{-1}$. Note that both of these conclusions continue to be true in the setting of Theorem~\ref{thm-fixedcoeffszero}. The conditions~\eqref{eq-whatsneedlow-1} and~\eqref{eq-whatsneedlow1} now follow from our assumptions on $\gamma_n(x)$.

\smallskip

We now bound the second two probabilities on the right-hand side of~\eqref{eq-sleptosplit}. Since $c_{1} > 0$ if $1 \in S$ and since $\chi_1(x) > 0$ if $1 \not\in S$, and because $\gamma_n(x) < 1$ for all $x \in \R$ and $n \in \mathbb{N}$, there exists a constant $C > 0$, depending only on $S$ and the values of the $c_i$, such that $\gamma_n^*(x) < C$ for all $x$ such that $|x| \in [(1-\xi_n)^{-1},\infty)$. Note that $C$ can be taken to be an absolute constant in the setting of Theorem~\ref{thm-fixedcoeffszero}. Thus, we have that
\begin{align}
 \P\big(\hat{f}_{n-k}^{b,*}(x) > \gamma_{n}^*(x),\,\forall x \in [(1-\xi_n)^{-1},\infty)\big)  & \geq \P\big(\hat{f}_{n-k}^*(x) > C,\,\forall x \in [(1-\xi_n)^{-1},\infty)\big),
\label{eq-probcn1}
\\
\P\big(\hat{f}_{n-k}^{b,*}(x) > \gamma_{n}^*(x),\,\forall x \in (-\infty,-(1-\xi_n)^{-1}]\big)  & \geq \P\big(\hat{f}_{n-k}^*(x) > C,\,\forall x \in (-\infty,-(1-\xi_n)^{-1}]\big),\label{eq-probcn2}
\end{align}
for all $n \gg 1$. The probabilities on the right-hand sides of~\eqref{eq-probcn1} and~\eqref{eq-probcn2} are of the form studied in~\cite[\S3]{MR1915821}, where they were each shown to be at least $n^{o(1)}$. This completes the proof of the lower bound.

\subsubsection{The case where $k$ is odd} \label{sec-lowkodd}

Now let $k$ be odd. Here, in addition to moving the first $k$ terms of the polynomial $f_n^b(x)$ to the right-hand side, we move the (random) term $a_{n-k-1}x^{n-k-1}$ to the right-hand side as well. As before, we restrict $a_{n-k-1}$ to lie within the interval $[1/2,1]$, and we replace $a_{n-k-1}$ with $\chi_{k+1}(x)$, and the rest of the argument proceeds exactly as in \S\ref{sec-lowkeven}.

\subsection{Upper bound} \label{sec-upper}

We now proceed with the proof of the upper bound. We start in \S\ref{sec-upperzero}, where we work out the case in which the fixed coefficients $c_i$ are all zero, and we finish in \S\ref{sec-uppernonzero} by explaining how to deduce the result for general $c_i$ from the case where they are all zero.

\subsubsection{The case where $c_i = 0$ for all $i \in S$} \label{sec-upperzero}

 In this section, we assume that $c_i = 0$ for all $i \in S$. In particular, we may assume without loss of generality that $1 \not\in S$ (recall, on the other hand, that $k \in S$ by assumption), and in particular, we are not in the setting of Theorem~\ref{thm-fixedcoeffszero}, and so we may keep $k$ fixed.

 \bigskip

 \noindent\emph{Notation}. We first set some notation. To bound the probability $\P\big(\hat{f}_{n,S}^{b}(x) > \gamma_{n}(x)\big)$, it is useful to obtain bounds on the covariance $\wt{c}_n$ of $\hat{f}_{n,S}^b$, which is given explicitly as follows:
\begin{align*}
\wt{c}_n(x, y) & = \frac{\mathbb{E}\big(f_{n,S}^b(x) f_{n,S}^b(y)\big)}{\sqrt{\mathbb{E}\big(f_{n,S}^b(x)^2\big) \mathbb{E}\big(f_{n,S}^b(y)^2\big)}} \\
& = \frac{(xy)^{n-k-1}r(xy) + \sum_{i = 1}^{n-k} (xy)^{n-k-i}}{\sqrt{\big(x^{2(n-k-1)}r(x^2) + \sum_{i = 1}^{n-k} x^{2(n-k-i)}\big)\big(y^{2(n-k-1)}r(y^2) + \sum_{i = 1}^{n-k} y^{2(n-k-i)}\big)}},
\end{align*}
where we have set $r(z) \defeq \sum_{\substack{i \in \{1, \dots , k\} \smallsetminus S}} z^{k + 1 - i}$. 
We will also have occasion to use the function $g$ defined by
$$g(x, y) = \frac{\lvert xy - 1\rvert}{\sqrt{\lvert (x^2-1)(y^2-1) \rvert}}.$$
Lastly, for the purpose of proving upper bounds, it suffices to restrict our attention to small subintervals of $\R$. To this end, fix $\delta \in (0,1/2)$, and define the four intervals $\mc{I}_1 \defeq [1 - n^{-\delta},1 - n^{-(1-\delta)}]$, $\mc{I}_2\defeq \mc{I}_1^{-1}$, $\mc{I}_3 \defeq -\mc{I}_2$, and $\mc{I}_4 \defeq -\mc{I}_1$. Set 
\begin{equation} \label{eq-defofv}
V \defeq \bigcup_{i=1}^4 \mc{I}_i \quad \text{and} \quad U \defeq \bigcup_{i = 1}^4 \mc{I}_i^2 \subset V^2.
\end{equation}

\bigskip
 
\noindent \emph{The proof}. Let $n$ be odd or even. Let $\alpha_n \in [0,n^{-\delta/2}]$, to be chosen later, and let $\chi_U \colon \R^2 \to \{0,1\}$ denote the indicator function of the subset $U \subset \R^2$. In~\cite[\S4]{MR1915821}, the upper bound $\P\big(\hat{f}_{n}^{b,*}(x) > \gamma_{n}(x)\big) \leq n^{-b+o(1)}$ is obtained in the case of Gaussian coefficients by constructing an auxiliary Gaussian process $\wt{f}_n(x)$ with covariance given by
\begin{equation} \label{eq-thecov}
    \frac{\big(1 - \alpha_n\big)}{g(x,y)} \chi_U(x,y) + \alpha_n.
\end{equation}
Let $c_n^*(x,y)$ denote the covariance of $\hat{f}_n^*(x) \defeq f_n^*(x)/\E\big(f_n^*(x)^2\big)$. It is shown in~\cite[Lemma 4.1]{MR1915821} that, for a suitable choice of $\alpha_n$, the covariance~\eqref{eq-thecov} is an upper bound on $c_n^*(x,y)$ for all $x,y \in V$, and that this upper bound specializes to an equality when $x = y$.  Thus, Slepian's lemma implies that to bound $\P\big(\hat{f}_{n}^{b,*}(x) > \gamma_{n}(x)\big)$ from above, the process $\hat{f}_n^{b,*}(x)$ can effectively be replaced with the process $\wt{f}_n(x)$, which turns out to be easier to work with.

An identical argument applies to prove that $\P\big(\hat{f}_{n,S}^{b}(x) > \gamma_{n}(x)\big) \leq n^{-b+o(1)}$, so long as we can show that the covariance $\wt{c}_n(x,y)$ is bounded above by~\eqref{eq-thecov}, with equality when $x = y$. We prove this as follows:
\begin{lemma} \label{lem-upbound}
     For $x,y \in V$, and for all $n \gg 1$, there exists $\alpha_n \in [0, n^{-\delta/2}]$ such that
    \begin{equation} \label{eq-ndbound}
    \wt{c}_n(x,y) \leq \frac{\big(1 - \alpha_n\big)}{g(x,y)} \chi_U(x,y) + \alpha_n,
    \end{equation}
    with equality when $x = y$.
\end{lemma}
\begin{proof}
First, assume $(x,y) \in U$. That we have equality when $x = y$ is obvious --- indeed, $\wt{c}_n(x,x) = g(x,x) = 1$. Consequently, we may assume that $x \neq y$ (in particular, we can divide by $x -y$, as we do multiple times in what follows). Now, by symmetry, it in fact suffices to take $x,y \in \mc{I}_1$ or $x,y \in \mc{I}_2$. We handle each of these cases separately as follows:

\smallskip
\noindent \emph{Case 1}: $x,y \in \mc{I}_1$. 
To start, note that the product $g(x,y)\wt{c}_n(x,y)$ may be rewritten as follows, writing $\wt{r}(z)$ for $(z - 1) r(z)$:
\begin{align}
    & g(x,y)\wt{c}_n(x,y) = \nonumber \\
    & \frac{|xy-1|}{\sqrt{|(x^2-1)(y^2-1)|}} \times \frac{(xy)^{n-k-1}r(xy) + \sum_{i = 1}^{n-k} (xy)^{n-k-i}}{\sqrt{\big(x^{2(n-k-1)}r(x^2) + \sum_{i = 1}^{n-k} x^{2(n-k-i)}\big)\big(y^{2(n-k-1)}r(y^2) + \sum_{i = 1}^{n-k} y^{2(n-k-i)}\big)}} = \nonumber \\
    & \on{sign}(xy-1) \times \frac{\wt{r}(xy) + 1 - (xy)^{k-n} }{\sqrt{\big(\wt{r}(x^2) + 1 - x^{2(k-n)}\big)\big(\wt{r}(y^2) + 1 - y^{2(k-n)}\big)}}.
    \label{eq-gcnrewrite}
\end{align}
Rearranging~\eqref{eq-ndbound} using~\eqref{eq-gcnrewrite}, and observing that $0 \leq g(x,y)\ll 1$ and $g(x,y)\wt{c}_n(x,y) = 1 + o(1)$ for $x,y \in \mc{I}_1$ we see that it suffices to prove the following bound:
    \begin{align} \label{eq-ndbound3}
n^{-\delta} & \gg \frac{\big(g(x,y)\wt{c}_n(x,y)\big)^2 - 1}{g(x,y)^2-1} = (1-x^2)(1-y^2)  \left[\frac{\wt{r}(xy)^2-\wt{r}(x^2)\wt{r}(y^2)}{(x-y)^2} +  \right. \\
&\qquad \left. \frac{-\wt{r}(x^2)x^2 + 2\wt{r}(xy)xy - \wt{r}(y^2)y^2}{(x-y)^2}   + \frac{-\wt{r}(y^2)x^{-2n} + 2\wt{r}(xy)x^{-n}y^{-n} - \wt{r}(x^2)y^{-2n}}{(x-y)^2}  + \left(\frac{x^{-n}-y^{-n}}{x-y}\right)^2 \right] \nonumber
    \end{align}
The first and second terms in the square brackets on the right-hand side of~\eqref{eq-ndbound3} are easily seen to be $\ll 1$ (where the implied constant depends on $k = \deg r$). The third and fourth terms are each $o(1)$ as $n \to \infty$, and the external factor of $(1-x^2)(1-y^2)$ is $\ll n^{-2\delta}$.

    \smallskip
\noindent \emph{Case 2}: $x,y \in \mc{I}_2$. Here, we prove the stronger result that
$g(x,y)\wt{c}_n(x,y) \leq 1$ (i.e., for this case, we can take $\alpha_n = 0$). Applying~\eqref{eq-gcnrewrite}, squaring both sides of this stronger inequality, cross-multiplying, and moving the terms to the left-hand side, we see that it suffices to prove that
\begin{align} \label{eq-nom}
    \big(\wt{r}(xy)+1\big)^2 - \big(\wt{r}(x^2)+1\big)\big(\wt{r}(y^2)+1\big) \leq 0.
\end{align}
Denote the left-hand side of~\eqref{eq-nom} by $L(x,y)$, and for any integer $m \geq k =  \deg r$, let $L_m(x,y)$ denote the same expression with $\wt{r}(z)$ replaced by $\wt{r}(z) + z^m$. Then it suffices to prove that $\wt{L}_m(x,y) \defeq L_m(x,y) - L(x,y) \leq 0$. We compute that
\begin{align}
    \frac{\wt{L}_m(x,y)}{(x-y)^2} & = x^{2m}y^{2m} + \left(\frac{x^m-y^{m}}{x-y}\right)^2-\left(\frac{x^{m+1}-y^{m+1}}{x-y}\right)^2 + \label{eq-line1} \\
    & \qquad\qquad\qquad\qquad\qquad  2\left(\frac{xy-1}{x-y}\right)^2r(xy)x^{m}y^{m} -\frac{(x^2-1)(y^2-1)}{(x-y)^2}\big(r(y^2)x^{2m}+r(x^2)y^{2m}\big). \nonumber
\end{align}
We start by estimating the first line of the right-hand side of~\eqref{eq-line1}. For $x,y \in \mc{I}_2$, the first term is $1 + o(1)$, the second is $m^2 + o(1)$, and the third is $-(m+1)^2 + o(1)$, making for a total of $-2m + o(1)$. As for the second line of~\eqref{eq-line1}, assume that $r(z) = z^e$, where $e \in \{1, \dots, k\}$. Under this assumption, the second line of~\eqref{eq-line1} may be conveniently reexpressed as follows:
\begin{equation} \label{eq-setereex}
x^{2k}y^{2k}\left[\left(\frac{x^{m-k+1}-y^{m-k+1}}{x-y}\right)^2 + x^{2}y^{2}\left(\frac{x^{m-k-1}-y^{m-k-1}}{x-y}\right)^2-(x^2y^2+1)\left(\frac{x^{m-k}-y^{m-k}}{x-y}\right)^2\right]
\end{equation}
For $x,y \in \mc{I}_2$, the quantity~\eqref{eq-setereex} the first term is $(m-k+1)^2 + o(1)$, the second is $(m-k-1)^2 + o(1)$, and the third is $-2(m-k)^2 + o(1)$, making for a total of $2 + o(1)$. Thus, when $r$ is a monomial, the second line of~\eqref{eq-line1} contributes $2 + o(1)$. Since this second line is linear in $r$, it follows that for any $r$, we get a contribution of at most $2k + o(1) \leq 2m-2 + o(1)$. We conclude that $\wt{L}_m(x,y) \leq (x-y)^2(-2 + o(1))$ for all $x,y \in \mc{I}_2$, so taking $n \gg k$ to be sufficiently large yields the desired inequality.

\bigskip
\noindent Finally, consider the case where $x,y \in V \smallsetminus U$. Without loss of generality, we may assume that $|x| \in \mc{I}_1$ and $|y| \in \mc{I}_2$. Then $y^{-n} \ll e^{-n^{1-\delta}}$ and $x^{-n} \gg e^{n^{\delta}}$, from which we can apply~\eqref{eq-gcnrewrite} to deduce that $|g(x,y)\wt{c}_n(x,y)| \ll e^{-n^\delta}$. Since $|g(x,y)| \ll n^{1-\delta}$, the desired bound follows.
\end{proof}
The upper bound  $\P\big(\hat{f}_{n,S}^{b}(x) > \gamma_{n}(x)\big) \leq n^{-b+o(1)}$ is then deduced from Lemma~\ref{lem-upbound} exactly as the upper bound $\P\big(\hat{f}_{n}^{b,*}(x) > \gamma_{n}(x)\big) \leq n^{-b+o(1)}$ is deduced from~\cite[Lemma~4.1]{MR1915821}; we omit the details for the sake of brevity. 

\subsubsection{The case of general $c_i$} \label{sec-uppernonzero}

We now take $c_i \in \R$ for each $i \in S$ in such a way that $c_1 > 0$ if $1 \in S$. Let $f_{n,0}(x) \defeq f_{n,S}^b(x) - \sum_{i \in S}c_ix^{n-i}$, and as usual, let $\hat{f}_{n,0}(x) \defeq f_{n,0}(x)/\E\big(f_{n,0}(x)^2\big)$. Define $\gamma_n^{**}(x)$ by
$$\gamma_n^{**}(x) \defeq \gamma_n(x)\sqrt{1+\frac{ \underset{{i \in  S}}{\sum} c_{i}^2x^{2(n-i)} + \underset{i \in \{1, \dots, k\} \smallsetminus S}{\sum} x^{2(n-i)}}{ \underset{i=1}{\overset{n-k}{\sum}} x^{2(n-k-i)}}}  -\frac{\underset{i \in S}{\sum} c_ix^{n-i}}{\sqrt{\underset{i = 1}{\overset{n-k}{\sum}}x^{2(n-k-i)}}}.$$
(Note the distinction between $\gamma_n^*(x)$, defined in~\eqref{eq-gn*}, and $\gamma_n^{**}(x)$.) 
Evidently the condition $\hat{f}_{n,S}^b(x) > \gamma_n(x)$ is equivalent to the condition $\hat{f}_{n,0}(x) > \gamma_n^{**}(x)$, so we have that
\begin{align*}
   \P\big(\hat{f}_{n,S}^{b}(x) > \gamma_{n}(x),\, \forall x \in \R\big) & = \P\big(\hat{f}_{n,0}(x) > \gamma_{n}^{**}(x),\, \forall x \in \R\big) \leq
   \P\big(\hat{f}_{n,0}(x) > \gamma_{n}^{**}(x),\, \forall x \in V\big).
\end{align*}
Now, it is shown in \S\ref{sec-upperzero} that
\begin{equation} \label{eq-limsupredv}
    \P\big(\hat{f}_{n,0}(x) > \gamma_{n}^{**}(x),\, \forall x \in V\big) \leq n^{-b+o(1)},
\end{equation}
as long as $|\gamma_n^{**}(x)| \to 0$ uniformly for $x \in V$. To verify that this condition holds, we imitate the bound in~\eqref{eq-boundsecterm} to find that
\begin{equation} \label{eq-boundsecterm2}
\sup\left\{\frac{|x^{n-j}|}{\sqrt{\sum_{i = 1}^{n-k} x^{2(n-k-i)}}}: x \in V\right\} = \frac{(1-n^{-\delta})^{-(n-j)}}{\sqrt{\sum_{i = 1}^{n-k}(1-n^{-\delta})^{-2(n-i)}}}\ll \frac{n^{-\delta/2}}{(1-n^{-\delta})^{k-j}}.
\end{equation}
Summing the bound~\eqref{eq-boundsecterm2} over $j \in \{1,\dots,k\}$ just as in~\eqref{eq-bipartite}, we obtain a bound of $kn^{-\delta/2}$, which converges to zero as $n \to \infty$. Note that this convergene holds even in the setting of Theorem~\ref{thm-fixedcoeffszero}, where $k = n^{o(1)}$.
\bigskip

This completes the proof of Theorem~\ref{thm-maingauss}, as well as the proof of Theorem~\ref{thm-fixedcoeffszero} in the case where the coefficient distributions are Gaussians.

\section{Proofs of Theorems~\ref{thm-main} and~\ref{thm-fixedcoeffszero} --- the case of general coefficients} \label{sec-pfgencoffs}

In~\cite[\S5]{MR1915821}, the authors use the  Koml\'os-Major-Tusn\'ady (henceforth, KMT) strong approximation theorem (see, e.g.,~\cite{MR0402883}) to extend their asymptotic formula for $P_n^*$ in the case of Gaussian coefficients to the case where the coefficient distribution is arbitrary (having zero mean, unit variance, and finite moments of all orders). The rough idea of their argument is to partition the random polynomial into several chunks, and to analyze the contributions of each chunk separately.

The purpose of this section is to deduce Theorems~\ref{thm-main} and~\ref{thm-fixedcoeffszero} (and hence also Theorem~\ref{thm-mainmon}) from Theorem~\ref{thm-maingauss} by means of an analogous method. The key difference is that one of the chunks of the polynomial contains all of the terms with the fixed coefficients. It is precisely to control the behavior of this chunk that we impose the ``niceness'' condition introduced in Definition~\ref{def-nice}. The following simple example demonstrates the necessity of the niceness condition:
\begin{example} \label{eg-probzero}
   Let the law of the $a_i$ be the uniform distribution from $[-\sqrt{3},\sqrt{3}]$, let $k = 2$, let $S = \{1,2\}$, and let $c_1 = 1$ and $c_2 = -2024$. Then, no matter what values are taken by the $a_i$, and for every $n \geq 2$, the polynomial $f_{n,S}(x)$ takes a negative value at $x = 2$. Indeed, we have in this case that
   $$f_{n,S}(2) \leq 2^{n-1} - 2024 \times 2^{n-2} + \sqrt{3} \times \big(2^{n-2}-1) \leq 2^{n-2} \times \big(2 - 2024 + \sqrt{3}\big) < 0.$$
\end{example}

\subsection{Approximation of coefficients by Gaussians via KMT} \label{sec-kmtcoupling}

We start by applying the KMT strong approximation theorem. In doing so, we shall make use of the following observation: for every strictly increasing sequence of nonnegative even integers $0 = k_0 < k_1 < \dots <  k_\ell $ and every $x \in [-1,1]$, the sequence
\begin{equation*}
\big\{x^{k_j} - x^{k_{j+1}} : j\in \{0,\ldots,\ell-1\}\big\} \cup \big\{x^{k_\ell}\big\}
\end{equation*}
forms a probability distribution (i.e., each term is nonnegative, and the sum of the terms is $1$). Thus, for any $s_0,\dots,s_k \in \R$, we have that
\be\label{able}
\left|s_0 + \sum_{j=1}^\ell (s_j-s_{j-1}) x^{k_j} \right|
= \left|s_\ell x^{k_\ell} +  \sum_{j=0}^{\ell-1} s_{j} (x^{k_j} - x^{k_{j+1}}) \right|
\leq \max_{0 \leq j \leq \ell} \big|s_{k_j}\big| \,.
\ee

Now, let $k \geq 1$, let $m \gg k$, and choose any subset $S \subset\{1,\dots,k\}$ that is either empty or contains $k$. Because $\E(a_i)=0$ and $\E(a_i^2)=1$ for each $i \in \mathbb{N}$, a double application of the KMT strong approximation theorem yields the following result: we can redefine the random variables $\big\{a_i : i \in \{0, \dots, m-1\} \text{ and } m-i \not\in S\big\}$ on a new probability space
with a corresponding sequence of
independent standard normal random variables $\big\{b_i : i\in \{0,\dots,m-1\} \text{ and } m-i \not\in S\big\}$
such that for any $p\geq 2$, some $\chi_p > 0$ that depends only on $p$, and all $t>0$, we have that
\be\label{sap}
\P\left(\max_{0 \leq j \leq \frac{n-1}{2}} \left|\sum_{\substack{ i \in \{0,\dots,j\} \\ n-2i \not\in S}} \big(a_{2i} - b_{2i}\big)\right|
\geq t\right) +
\P\left(\max_{0 \leq j \leq \frac{n-3}{2}} \left|\sum_{\substack{i \in \{0,\dots,j\} \\ n-(2i+1) \not\in S}} \big(a_{2i+1} -b_{2i+1}\big)\right|
 \geq t\right) \leq \chi_p n \E|a_0|^p t^{-p} .
\ee
Note that we can express the right-hand side of~\eqref{sap} solely in terms of $a_0$ because the coefficients $a_i$ are identically distributed. By the triangle inequality, we have
\be
\label{able2}
\left|\sum_{\substack{i \in \{0,\dots, m-1\} \\ m-i \not\in S}} (a_i - b_i) x^i\right| \leq
\left|\sum_{\substack{i \in \{0,\dots, \lfloor \frac{m-1}{2} \rfloor\} \\ m-2i \not\in S}} (a_{2i}-b_{2i}) x^{2i}\right|
+\left|\sum_{\substack{i \in \{0,\dots,\lfloor \frac{m-2}{2} \rfloor\} \\ m-(2i+1) \not\in S}} (a_{2i+1}- b_{2i+1})
x^{2i+1}\right|
\ee
so applying \eqref{able} twice --- first taking $s_j=\sum_{\substack{i \in \{0,\dots,j\} \\ m-2i \not\in S}} (a_{2i}-b_{2i})$, and then taking
$s_j=\sum_{\substack{i \in \{0,\dots, j\} \\ m-(2i+1) \not\in S}}(a_{2i+1}-b_{2i+1})$ --- and combining the result with~\eqref{sap} and~\eqref{able2} yields that for all $m \leq n$, we have
\be\label{sapc1}
\P\left(\sup_{|x| \leq 1} \left|f_{m,S}(x)-f_{m,S}^b(x)\right| \geq  2 t \right) \leq \chi_p n \E|a_0|^p t^{-p},
\ee
where $f_{m,S}^b$ is defined by
$$f_{m,S}^b(x) \defeq \sum_{i \in S} c_i x^{m-i} + \sum_{\substack{i \in \{1, \dots, m\} \\ i \not\in S}} b_{m-i}x^{m-i}.$$

Notice that~\eqref{sapc1} only covers the region where $|x| \leq 1$. To handle the region $|x| \geq 1$, we invert $x$ and consider the random polynomial in reverse. Indeed, define
$$
g_{m,S}(x) \defeq x^{m-1} f_{m,S}(x^{-1}) = \sum_{i\in S} c_{i} x^{i-1} + \sum_{\substack{i \in \{1, \dots, m\} \\ i \not\in S}} a_{m-i} x^{i-1} \;,
$$
and similarly define $g_{m,S}^b(x) \defeq x^{m-1} f_{m,S}^b(x^{-1})$. By an analogous argument, we also obtain for all $m \leq n$ the bound
\be\label{sapc2}
\P\left(\sup_{|x| \leq 1} \left|g_{m,S}(x)-g_{m,S}^b(x)\right| \geq  4t\right) \leq \chi_p n \E|a_0|^p t^{-p}.
\ee
Indeed, bounding $\big|g_{m,S}-g_{m,S}^b\big|$ amounts to replacing $(a_i,b_i)$ with
$(a_{m-1-i},b_{m-1-i})$ --- which leads us to take $s_j = \sum_{\substack{i \in \{0,\dots, j\} \\ m-2i \not\in S}} (a_{m-1-2i}-b_{m-1-2i})$ first, and to take next $s_j = \sum_{\substack{i \in \{0,\dots, j\} \\ m-(2i+1) \not\in S}} (a_{m-1-(2i+1)}-b_{m-1-(2i+1)})$. Note that the approximation error doubles from $2t$ to $4t$ because the partial sums $s_j$ are now computed in reverse; for instance, we have
$$\max_{0 \leq j \leq \lfloor \frac{m-1}{2} \rfloor} \left|\sum_{\substack{ i \in \{0,\dots,j\} \\ 2i+1 \not\in S}} \big(a_{m-1-2i} - b_{m-1-2i}\big)\right| \geq 2t \,\,\Longrightarrow\,\, \max_{0 \leq j \leq \lfloor\frac{m-1}{2}\rfloor} \left|\sum_{\substack{ i \in \{0,\dots,j\} \\ m-2i \not\in S}} \big(a_{2i} - b_{2i}\big) \right|  \geq t.$$

\subsection{Notation and setup} \label{sec-allthens}

In this section, we introduce the notation necessary for us to partition the random polynomial into pieces and complete the proof of Theorem~\ref{thm-main}. Let $n,\,k,\,S$ be as in the setting of the theorem, and for $m \gg k$, define the following quantities
$$\sigma_{m,S}(x)\defeq \sqrt{\E\big(f_{m,S}(x)^2\big)} \quad \text{and} \quad \wt{\sigma}_{m,S}(x) \defeq \sqrt{\E\big(g_{m,S}(x)^2\big)},$$
and following~\eqref{eq-thehats}, define the normalized random polynomials
\begin{align*}
& \hat{f}_{m,S}(x)\defeq \frac{f_{m,S}(x)}{\sigma_{m,S}(x)}\quad \text{and} \quad \hat{f}^b_{m,S}(x) \defeq \frac{f^b_{m,S}(x)}{\sigma_{m,S}(x)}; \\
& \hat{g}_{m,S}(x)\defeq \frac{g_{m,S}(x)}{\wt\sigma_{m,S}(x)}\quad \text{and} \quad \hat{g}^b_{m,S}(x) \defeq \frac{g^b_{m,S}(x)}{\wt\sigma_{m,S}(x)}.
\end{align*}
Following the convention in \S\ref{sec-theintro}, we write $f_m \defeq f_{m,\varnothing}$, $g_m \defeq g_{m,\varnothing}$, and $\sigma_m \defeq \sigma_{m,\varnothing} = \wt{\sigma}_{m,\varnothing}$.
Just as in~\cite[Equation (5.6)]{MR1915821}, we introduce the following list of $n$-dependent quantities and functions:
\be \label{eq-thelistofall}
\begin{array}{lll}
p_n:&p_n\uparrow \infty, \chi_{p_n} \E|a_0|^{p_n}\leq n&
\chi_{p_n} \,\mbox{\rm is the KMT constant in \eqref{sap}}\\
{\epsilon}_n:&{\epsilon}_n\downarrow 0, {\epsilon}_n\geq
\max\{20/p_n,(\log n)^{-1/2}\},&{\epsilon}_n\,\, \mbox{\rm is the
 smallest possible value}\\
& \,\,\,\,\,\,\,\,\,\,\,\,\,\,\,\,\,\,\,\,\,\,\,\,\,\,\,2n^{3{\epsilon}_n}=2^j\,\, \mbox{\rm for some $j \in \mathbb{N}$}&
 \mbox{\rm
  \,\,\,\,\,\,\,\,\,\,\,\,\,\, satisfying the constraints
}\\
m_n:&m_n\to\infty, m_n=2n^{3{\epsilon}_n}& \,\,
\mbox{\rm }\\
\ol{\gamma}_n:&
\ol{\gamma}_n(x)=\max \{0,\gamma_n(x),\gamma_n(x^{-1})\}&
\,\,
\mbox{\rm }\\
\ol{\sigma}_n:&
\ol{\sigma}_n(x)=\max \{\sigma_{n,S}(x),\wt{\sigma}_{n,S}(x)\}&
\,\,
\mbox{\rm }\\
\rho_n:&\rho_n\to 0,
\rho_n = \sup_{|x| \leq 1-m_n^{-1}} \; \{ \ol\sigma_n(x) \ol{\gamma}_n(x) \}&
 \rho_n\leq c n^{-\delta/2},\,\, \mbox{\rm for
some $c>0$}\\
r_n:&c n^{-\delta/2}& \mbox{\rm for $n\geq 3m_n$; $c$ is as in bound on
$\rho_n$}\\
\xi_n:&\xi_n(x)=6x^m \sigma_{n-2m}(x)\ol{\gamma}_n(x)&
\end{array}
\ee
Partition the interval $[-1,1]$ into
$\mc{I} \defeq \{ x : |x| \geq 1-0.5 n^{- {\epsilon_n}} \}$ and $\mc{I}^c \defeq [-1,1]
\smallsetminus \mc{I}$. Then the constant $6$ in the definition of the function $\xi_n(x)$ in~\eqref{eq-thelistofall} is chosen to ensure that, for all $n \gg 1$, we have $\xi_n(x) \geq \ol{\sigma}_n(x) \ol{\gamma}_n(x)$ for all $x \in \mc{I}$ such that $|x| \geq 1 - m_n^{-1}$. Since $r_n \geq \ol{\sigma}_n(x)\ol{\gamma}_n(x)$ for all $x$ such that $|x| \leq 1 - m_n^{-1}$, it follows that $2r_n + \xi_n(x) \geq \ol{\sigma}_n(x) \ol{\gamma}_n(x)$ for all $n \gg 1$ and $x \in \mc{I}$.
\smallskip

Next, let
$f_{n,S}=f_{n,S}^L+f_{n,S}^M+ f_{n,S}^H$ where
\be \label{LMH}
f_{n,S}^L(x) \defeq \sum_{i=0}^{m_n-1} a_i x^i, \qquad
f_{n,S}^M(x) \defeq \sum_{i=m_n}^{n-(m_n+1)} a_i x^i, \qquad
f_{n,S}^H(x) \defeq \sum_{i \in S} c_ix^{n-i} + \sum_{\substack{i \in \{1, \dots, m_n\} \\ i \not\in S}} a_{n-i} x^{n-i}.
\ee
Similarly, we let $g_{n,S}=g_{n,S}^L+g_{n,S}^M+ g_{n,S}^H$, where for each $\bullet \in \{L,M,H\}$ we define $g_{n,S}^{\bullet}(x) \defeq x^{n-1}f_{n,S}^{\bullet}(x^{-1})$.

\subsection{Lower bound}

We first handle the lower bound in Theorem~\ref{thm-main}, leaving the upper bound to \S\ref{sec-uppergens}. With notation as in \S\ref{sec-allthens}, we have the following chain of implications for all $n \gg 1$:
\begin{eqnarray*}
\hat{f}_{n,S}(x) > \gamma_n(x) ,\, \forall x \in \R 
& \Longleftarrow & \hat{f}_{n,S}(x) >\ol{\gamma}_n(x) ,\, \hat{g}_{n,S}(x) > \ol{\gamma}_n(x), \, \forall  x \in [-1,1] \\
& \Longleftarrow & f_{n,S}^M(x) > \xi_n(x),\; g_{n,S}^M(x) > \xi_n(x) : \forall x \in \mc{I}, \quad \text{and} \\
&& \qquad f_{n,S}^M(x) >-r_n,\; g_{n,S}^M(x) >-r_n : \forall  x \in \mc{I}^c,\quad \text{and}\\
&& \qquad\qquad f_{n,S}^L(x) >3r_n, \; g_{n,S}^L(x) \geq -r_n : \forall x \in [-1,1],\quad \text{and}\\
&& \qquad\qquad\qquad f_{n,S}^H(x) \geq -r_n, \; g_{n,S}^H(x) > 3r_n : \forall x \in [-1,1].
\end{eqnarray*}
Since the polynomial pairs $\big(f_{n,S}^L,g_{n,S}^L\big)$, $\big(f_{n,S}^M,g_{n,S}^M\big)$ and
$\big(f_{n,S}^H,g_{n,S}^H\big)$ are mutually independent, \mbox{it follows that}
\begin{align}
& P_{n,S,\gamma_n} =
\P\big( \hat{f}_{n,S}(x) >\gamma_n(x),\, \forall x \in \R \big) \geq \nonumber  \\
& \quad \P\Big(
 \big\{f_{n,S}^M (x) >\xi_n(x), \; g_{n,S}^M(x) >\xi_n(x),
\, \forall \ x \in \mc{I}\big\} \cap 
\big\{f_{n,S}^M(x) >-r_n, \; g_{n,S}^M(x) > -r_n, \, \forall  x \in \mc{I}^c\big\} \Big) \times \label{eq-thesplitters} \\
&\quad \,\,\,\, \P\big( f_{n,S}^L(x) >3r_n, \, g_{n,S}^L(x) \geq -r_n ,\, 
\forall x \in [-1,1] \big) \times  \P\big( f_{n,S}^H(x) \geq -r_n, \, g_{n,S}^H(x) > 3r_n ,\,
\forall x \in [-1,1] \big). \nonumber
\end{align}
Call the three factors in~\eqref{eq-thesplitters} $\mc{Q}_1$, $\mc{Q}_2$, and $\mc{Q}_3$, respectively. The factor $\mc{Q}_1$ was estimated in~\cite[\S5, (5.9)--(5.10)]{MR1915821} using the strong approximation results~\eqref{sapc1} and~\eqref{sapc2} in the case $S = \varnothing$, and a lower bound of 
\begin{equation} \label{eq-curlq1bound}
\mc{Q}_1 \geq n^{-b+o(1)}
\end{equation}
was obtained. The factor $\mc{Q}_2$ was also estimated in~\cite[\S5, paragraphs immediately following (5.10)]{MR1915821}, where a lower bound of 
\begin{equation} \label{eq-curlq2bound}
\mc{Q}_2 \geq n^{o(1)}
\end{equation}
was obtained. (To be clear, in \emph{loc.~cit.}, the probabilities $\mc{Q}_i$ are expressed as $\mc{Q}_1 = Q_1 - 2Q_2$ and $\mc{Q}_2 = Q_3 - Q_4$, where each $Q_i$ is a certain probability, and they bound the $\mc{Q}_i$ by controlling each of the $Q_i$ separately.)

As for bounding $\mc{Q}_3$, define the following probabilities:
\begin{align*}
    \mc{Q}_3' & \defeq \P\left(g_{m_n,S}(x) > 3r_n, \forall x \in [-1,1]; \text{ and } x^{n-m_n}f_{m_n,S}(x) \geq -r_n, \forall x \in \pm[1-m_n^{-1},1]\right),\\
    \mc{Q}_4' & \defeq \P\left( x^{n-m_n} f_{m_n,S}(x) \leq -r_n,\, \mbox{for some} \ x \text{ such that } |x| \leq 1 - m_n^{-1}\right).
\end{align*}
Then $\mc{Q}_3 \geq \mc{Q}_3' - \mc{Q}_4'$, so it suffices to determine lower and upper bounds on $\mc{Q}_3'$ and $\mc{Q}_4'$, respectively. To bound $\mc{Q}_4'$, recall that $m_n=2 n^{3{\epsilon}_n} $ and ${\epsilon}_n \to 0$. It then follows that
for all $n \gg 1$, we have
\begin{equation} \label{eq-q4firstbound}
\mc{Q}_4' \leq \P\left( \sup_{|x| \leq 1 - m_n^{-1}} |x|^{n-m_n} \big|f_{m_n}(x)\big| \geq r_n \right)
\leq \P\left( \sup_{|x| \leq 1 - m_n^{-1}} \big|f_{m_n}(x)\big|\geq  e^{\sqrt{n}}\right).
\end{equation}
To proceed, we make use of the following lemma (this is~\cite[Lemma~5.1]{MR1915821}; see \cite{MR0455094} for a proof).
\begin{lemma}\label{moment}
Let $\{T_x : x \in [a, b]\}$ be an almost surely continuous
stochastic process with $T_a=0$.
Assume that $\E\big|T_x-T_y\big|^2 \leq K (x-y)^2$ for all $x, y \in [a,b]$. 
Then, we have that
$$
\E\left( \sup_{x \in [a,b]} T_x^2 \right) \leq 4 K (b-a)^2.
$$
\end{lemma}

Take $T_x=f_{m_n,S}(x)-f_{m_n,S}(0) = f_{m_n,S}(x) - a_0$. Then, just as in~\cite[display immediately preceding Lemma 5.1]{MR1915821}, we have the bound
\begin{equation} \label{eq-expectsqubounds}
\E\big|T_x-T_y\big|^2 = \E\big|f_{m_n,S}(x) - f_{m_n,S}(y)\big|^2 \leq n^3(x-y)^2.
\end{equation}
Using~\eqref{eq-expectsqubounds} and applying Lemma~\ref{moment} two times, once with $a=0,\,b = 1 - m_n^{-1}$ and once with $a = -1 + m_n^{-1}, b = 0$, we find that
\begin{equation} \label{eq-expectbounds}
   \E\left( \sup_{|x| \leq 1 - m_n^{-1}} T_x^2 \right) \leq 8 n^3 \big(1 - m_n^{-1}\big)^2.
\end{equation}
Upon combining~\eqref{eq-q4firstbound} with~\eqref{eq-expectbounds} and applying Markov's inequality multiple times, it follows that for all $n \gg 1$ we have
\begin{equation} \label{eq-q4bound}
\mc{Q}_4' \leq \P\big(|a_0| \geq e^{\sqrt{n}}\big) + \P\left(\sup_{|x| \leq 1 - m_n^{-1}} T_x^2 \geq e^{2\sqrt{n}}\right) \ll e^{-\sqrt{n}}
 + 8n^3 \big(1 - m_n^{-1}\big)^2e^{-2\sqrt{n}} \ll e^{-n^{1/3}}.
\end{equation}

We now turn our attention to bounding $\mc{Q}_3'$. For this, we prove the next lemma, from which one readily deduces that
\begin{equation}
\label{eq-q3bound}
\mc{Q}_3' \gg n^{-O(\epsilon_n)},
\end{equation}
where the implied constants are fixed.
\begin{lemma} \label{lower}
There exists $c<\infty$ such that for all $m=2^{\kappa+1}$, where $\kappa \in \mathbb{N}$ is sufficiently large, we have
\be
\P\Big( g_{m,S}(x) > m^{-2},\, \forall x \in [-1,1];  \text{ and }\,
x f_{m,S}(x) \geq 0,\, \forall x \in \pm [1-2^{-\kappa},1] \Big)
\geq m^{-c}.
\label{i3}
\ee
\end{lemma}
\begin{remark}
    The statement and proof of Lemma~\ref{lower} draw heavy inspiration from~\cite[Lemma~5.2]{MR1915821}, except that the roles of $f$ and $g$ are reversed. This distinction matters not in \emph{loc.~cit.}~but is of particular importance in the present article, where we use the niceness condition to prove the lemma.
    \end{remark}
\begin{proof}[Proof of Lemma~\ref{lower}]
Define the intervals
$\mc{J}_{j}\defeq \{ x : 1-2^{-j} \leq |x| \leq 1-2^{-j-1} \}$ for $j \in \{1,\ldots,\kappa-1\}$
and $\mc{J}_{\kappa}=\{ x : 1-2^{-\kappa} \leq |x| \leq 1 \}$. Let $1 \ll s \leq \kappa$, and further define $\mc{J}_0 \defeq \{ x : |x| \leq 1-2^{-s} \}$. We will decompose $f_{m,S}$ and $g_{m,S}$ into sums over
$\kappa-s+2$
terms involving polynomials $f^j$ and $g^j$ of smaller degree, where $j \in \{0,s,s+1,\ldots,\kappa\}$.
Specifically, we
write
$$
f_{m,S}(x) =x^{m-2^s} f^0 (x) + \sum_{j = s}^\kappa  x^{m-2^{j+1}}  f^j(x) \quad \text{and} \quad g_{m,S}(x) =g^0 (x) + \sum_{j=s}^{\kappa} x^{2^j} g^j(x),$$
where the $f^j$ and $g^j$ are completely determined by the conditions $\deg f^0 = 2^{s}-1 = \deg g^0$ and $\deg f^j = 2^j - 1 = \deg g^j$ for each $j \in \{s, \dots, \kappa\}$.

For some sufficiently large $\Gamma_0 > 0$ that does not depend on $\kappa$, we have
\begin{equation}
\label{v22}
 2^{j/2} (\Gamma_0-1) x^{2^j} -
\sum_{\substack{i=s \\ i \neq j}}^{\kappa} 2^{i/2} x^{2^i} \geq 0\; , \quad
\quad \forall x \in \mc{J}_j, \;\; j \in \{s,\ldots,\kappa\}.
\end{equation}
Moreover, since $\E(a_0)=0$ and $\E(a_0^2)=1$, there exist real numbers $\alpha > \beta > 0$ such that $\P(|a_0-\alpha| \leq \beta')>0$ for every $\beta' \in (0,\beta)$. Fixing such an $\alpha$, let $M \in \mathbb{N}$ be an integer to be chosen later, and take $s \gg 1$ to be sufficiently large, so that
\begin{equation}
\label{v3}
\sum_{i=s}^\infty 2^{i/2} x^{2^i} \leq \frac{\alpha}{2M} , \quad \quad \forall x \in \mc{J}_0.
\end{equation}
Note that such an $s$ always exists because, for every $x \in \mc{J}_0$, the sum on the left-hand side of~\eqref{v3} converges.

The sets $\mc{J}_0,\mc{J}_s,\mc{J}_{s+1},
\ldots,\mc{J}_\kappa$ form a partition of the interval $[-1,1]$. One checks that for $\kappa$ large enough,
we have
\begin{equation}
\label{v1}
m^{-2} \leq
\min\bigg\{\frac{\alpha}{2M},\inf_{\substack{x \in \mc{J}_j, \\ j \in \{s,\ldots,\kappa\}}}\, 2^{j/2} x^{2^j}\bigg\}
\end{equation}
Combining~\eqref{v22} with~\eqref{v1}, we deduce that
\begin{align}
\big\{ g_{m,S}(x) > m^{-2},\, \forall x \in [-1,1] \big\} & \supset
\bigcap_{j=s}^\kappa
\big\{ g^j(x) > 2^{j/2}\Gamma_0 ,\, \forall x \in \mc{J}_j; \text{ and } g^j(x) \geq -2^{j/2},\, \forall x \in \mc{J}^c_j \big\} \nonumber \\
&\qquad\qquad\quad \cap
\bigg\{ g^0(x) > {\frac{\alpha}{M}},\, \forall  x \in \mc{J}_0; \text{ and }
 g^0(x) \geq 0,\, \forall x \in \mc{J}^c_0 \bigg\} \label{dec1}
\end{align}
Note that for all $x \in [-1,1]$, we have
\be
\label{dec2}
\big\{ x f_{m,S}(x) \geq 0 \} \supset  \big\{ x f^0(x) \geq -2^{\kappa/2} \big\} \cap
\bigcap_{j=s}^{\kappa} \big\{ x f^j(x) \geq 2^{j/2}\Gamma_0 \big\}.
\ee

The polynomial pairs $(f^j,g^j)$ for $j \in \{0,s,\ldots,\kappa\}$ are mutually
independent of each other, with the pair $(f^j,g^j)$ obeying
the same law as that of $(f_{2^j},g_{2^j})$ for each $j \neq 0$. Consequently, combining~\eqref{dec1} and~\eqref{dec2} yields that
\begin{align}
&\mathbb{P}\big( g_{m,S}(x) > m^{-2},\, \forall x \in [-1,1]; \text{ and }
x f_{m,S}(x) \geq 0,\,  \forall x \in \mc{J}_{\kappa} \big) \geq \eta_{s,\kappa} \times \prod_{j = s}^\kappa q_j, \label{qn1}
\end{align}
where the factors $\eta_{s,\kappa}$ and $q_j$ are defined as follows:
\begin{align*}
\eta_{s,\kappa} & \defeq \mathbb{P}\bigg( g_{2^s,S}(x) > \frac{\alpha}{M},\, \forall x \in \mc{J}_0; \text{ and }
 g_{2^s,S}(x) \geq 0,\, \forall x \in \mc{J}^c_0; \text{ and }
 x f_{2^s,S}(x) \geq -2^{\kappa/2},\, \forall x \in \mc{J}_\kappa \bigg) \nonumber \\
q_j &\defeq 
\mathbb{P}\Big( f_{2^j} (x) > \Gamma_0 2^{j/2},\, \forall x \in \mc{J}_j; \text{ and }
 f_{2^j}(x) \geq -2^{j/2},\, \forall  x \in \mc{J}^c_j; \text{ and }
 x g_{2^j}(x) \geq  \Gamma_0 2^{j/2},\, \forall x \in \mc{J}_{\kappa} \Big)
\end{align*}
That $q_j$ is uniformly bounded away from zero in $\kappa$, independently of $j$, was established in~\cite[Proof of Lemma~5.2]{MR1915821} --- indeed, note that $q_j$ arises from the pieces of the random polynomials $f_{m,S}$ and $g_{m,S}$ that do not involve the fixed coefficients $c_i$ for $i \in S$. It remains to show that $\eta_{s,\kappa}$ is uniformly bounded away from zero in $\kappa$, for all sufficiently large $\kappa$. By our assumption that the coefficient data is nice, there exists an even integer $s' > k$, a set of values $\{v_i : i \in \{0,\dots, s'-1\} \smallsetminus S\} \subset \R$, and a small constant $\ol{\beta} > 0$ such that $\P\big(|a_i - v_i| \leq \ol{\beta}\big) > 0$ for each $i \in \{0,\dots, s'-1\}$ and such that whenever $v_i' \in [v_i - \ol{\beta}, v_i + \ol{\beta}]$ for every $i \in \{0, \dots, s'-1\} \smallsetminus S$, we have 
\begin{equation} \label{eq-recallcond}
\sum_{i \in S} c_i x^{i-1} + \sum_{\substack{i \in \{0, \dots, s'-1\} \\ i \not\in S}} v_i' x^i > 0,\quad \forall x \in (-1,1). 
\end{equation}
Indeed, the niceness assumption guarantees that the left-hand side of~\eqref{eq-recallcond} is nonzero for all $x \in (-1,1)$ and $v_i' \in [v_i - \ol{\beta}, v_i + \ol{\beta}]$, and so the fact that $c_1 > 0$ ensures that it is actually positive. For any $\ell > s'/2$, consider the sum of terms
\begin{equation} \label{eq-tackon}
v_{s'}' x^{s'} + v_{s'+1}'x^{s'+1} + \cdots + v_{2\ell}'x^{2\ell} + v_{2\ell+1}' x^{2\ell+1}.
\end{equation}
If we restrict the coefficients in~\eqref{eq-tackon} to satisfy 
\begin{equation} \label{eq-decreaser}
\alpha + \beta \geq v_{2i}' \geq v_{2i+1}' \geq \alpha - \beta, \quad \text{for every $i \in \{s'/2,\dots,\ell\}$,}
\end{equation}
then the sum in~\eqref{eq-tackon} is nonnegative for all $x \in [-1,1]$. Now, choose $s$ so that $2^s > s'$, and take $\ell = 2^{s-1}-1$. We claim that, if the coefficients $v_i'$ satisfy $v_i' \in [v_i - \ol{\beta}, v_i + \ol{\beta}]$ for every $i \in \{0, \dots, s'-1\} \smallsetminus S$ along with~\eqref{eq-decreaser}, then we have
\begin{equation} \label{eq-moddedcond}
\sum_{i \in S} c_i x^{i-1} + \sum_{\substack{i \in \{0, \dots, 2^s-1\} \\ i \not\in S}} v_i' x^i > \frac{\alpha}{M},\quad \forall x \in [-1 + 2^{-s}, 1 - 2^{-s}]
\end{equation}
for sufficiently small $\beta$ and large $M$. Indeed, there exists some $\xi \in (0,1-2^{-s})$ such that the left-hand side of~\eqref{eq-recallcond} is at least $c_1/2$ for all $x$ such that $|x| \leq \xi$. Then for all $x$ such that $\xi \leq |x| \leq 1 - 2^{-s}$, we have that
\begin{equation}
\sum_{i = s'}^{2^s-1} \alpha x^i = x^{s'} \times \sum_{i = 0}^{2^s - s' - 1} \alpha x^i  \geq \frac{\xi^{s'}}{4} \times \alpha.
\end{equation}
Note that the quantities $c_0/2$ and $\xi^{s'}/4$ are fixed. Thus, if we take $\beta$ to be sufficiently small, then the sum on the left-hand side of~\eqref{eq-moddedcond} is bounded away from zero by some fraction of $\alpha$, proving the claim. Finally, with the coefficients $v_i'$ chosen as above and for $\beta$ sufficiently small, we also have that
\begin{equation} \label{eq-nowthefs}
\left|xf_{2^s}(x)\right| = \left|\sum_{i \in S} c_i x^{2^s-i} + \sum_{\substack{i \in \{0, \dots, 2^s-1\} \\ i \not\in S}} v_i' x^{2^s-1-i}\right| > 2^sC \times \alpha,\quad \forall x \in \mc{J}_k,
\end{equation}
where $C > 0$ depends on the coefficients $c_i,\, v_i$ as well as on $\beta,\, \ol{\beta}$. Taking $\kappa$ to be so large that $2^{\kappa/2} > 2^sC \times \alpha$, we conclude that $\eta_{s,\kappa}$ is uniformly bounded away from zero in $\kappa$.

It remains to explain how the proof of the lemma adapts to the setting of Theorem~\ref{thm-fixedcoeffszero}. Because of our assumption that $k(n) = n^{o(1/\sqrt{\log n})}$, the only aspects of the proof that change are the choice of $s = s(n)$, which must be taken so that $2^{s(n)} \ggg k(n)$, as well as the estimation of the factor $\eta_{s(n),\kappa}$. But because the subleading fixed coefficients $c_2, \dots, c_{k(n)}$ are all equal to zero, it is easy to see that the coefficient data are nice, in the sense that a statement analogous to~\eqref{eq-recallcond} holds. The rest of the proof is the same.
\end{proof}

Combining the bounds on $\mc{Q}_3'$ and $\mc{Q}_4'$ obtained in~\eqref{eq-q3bound} and~\eqref{eq-q4bound}, respectively, we deduce that
\begin{equation} \label{eq-curlq3bound}
    \mc{Q}_3 \geq \mc{Q}_3' - \mc{Q}_4' \geq n^{-o(1)}.
\end{equation}
The lower bound $P_{n,S,\gamma_n} \geq n^{-b+o(1)}$
in Theorem \ref{thm-main}, as well as the lower bound $P_{n,S_n} \geq n^{-b + o(1)}$ in Theorem~\ref{thm-fixedcoeffszero},
then follows by substituting the bounds~\eqref{eq-curlq1bound},~\eqref{eq-curlq2bound}, and~\eqref{eq-curlq3bound} on the factors $\mc{Q}_1$, $\mc{Q}_2$, and $\mc{Q}_3$ into the right-hand side of~\eqref{eq-thesplitters}.

\subsection{Upper bound} \label{sec-uppergens}

We finish by handling the upper bound in
Theorem \ref{thm-main}. Let $\eta_n \defeq \inf\{ \gamma_n(x) : ||x|-1|\leq n^{-{\epsilon_n}}
 \}$, and recall that our assumptions imply that $\eta_n \to 0$. Then, by applying the strong approximation results~\eqref{sapc1} and~\eqref{sapc2} with $m = n$, $t = n^{\epsilon_n/4}$, and $p = p_n$, we deduce that
\begin{align*}
P_{n,S,\gamma_n} &=
\P\big( \hat{f}_{n,S}(x) >\gamma_n(x), \, \forall x \in \R \big) \leq \P\big(
\hat{f}_{n,S}(x) >\eta_n \text{ and } \hat{g}_{n,S}(x) >\eta_n, \, \forall x \in \mc{I} \big)  \\
&\qquad\qquad \leq
\P\big( \hat{f}^b_{n,S}(x) > \eta_n - n^{-{\epsilon_n}/4} \text{ and } \hat{g}_{n,S}^b(x) >\eta_n  -n^{-{\epsilon_n}/4},
\, \forall x \in \mc{I} \big) + 2 n^{-3} \leq n^{-b+o(1)}
\end{align*}
for all $n \gg 1$. The second-to-last inequality follows because $\inf_{x \in \mc{I}} \{\sigma_{n,S}(x),\wt{\sigma}_{n,S}(x)\} \geq n^{\epsilon_n/2}$ for all $n \gg 1$; the last inequality follows by an application of Theorem \ref{thm-maingauss}, taking the threshold to be $\eta_n-n^{-{\epsilon}_n/4}$ when $x \in \mc{I} \cup \mc{I}^{-1}$
and $-\infty$ for all other $x$. This completes the proof of Theorem~\ref{thm-main}.

\section{Proof of Theorem~\ref{thm-locmaingen} and~\ref{thm-fixedcoeffszero2}} \label{sec-thelast}

The purpose of this section is to prove Theorems~\ref{thm-locmaingen} and~\ref{thm-fixedcoeffszero2} (and hence also Theorem~\ref{thm-locmainmon} and Corollary~\ref{cor-algint}). The arguments given in this section are logically independent of the previous two sections, so in particular, we obtain a second proof of Theorem~\ref{thm-mainmon}. We retain notation ($n$, $k$, $S$, and $j$) as in the setting of the theorem.

\subsection{Lower bound}

To obtain the lower bound, we split into cases according to the parity of $n,j$. In \S\ref{sec-lowkodd1}, we treat the case where $n-1,j$ are even, and in \S\ref{sec-lowkeven1}, we handle the case where $n-1,j$ are odd. 

\subsubsection{The case where $n-1,j$ are even} \label{sec-lowkodd1}

Let $n-1,j$ be even. Note that by the assumptions of the theorem, the support of the distribution of the random coefficients intersects both intervals $(-\infty, 0)$ and $(0, \infty)$. Further, there exists a sequence $(p_n)_n$ with $\lim_{n \to \infty} p_n = \infty$ and $\E(|a_0|^{p_n}) \le n$. Set $\rho_n \defeq \max\{5p_n^{-1}, (\log n)^{-1/2}\}$; note that the definitions of $p_n$ and $\rho_n$ are \emph{not} the same as in \S\ref{sec-allthens}. 
Let $M \gg 1$ be such that $\mathbb{P}(|a_i| < M) > 0$. Define the set 
$$\mathcal{V} \defeq \{v \in \R : \mathbb{P}(|a_i-v|<\epsilon) > 0,\, \forall \epsilon > 0\}\subseteq \R,$$ and observe that $\mathcal{V} \neq \varnothing$. We begin with the following result, a somewhat weaker version of which is claimed and proven in~\cite[\S8.1]{MR1915821}. 
\begin{lemma} \label{lem-hconds}
There exist a constant $C_j$ $($depending only on $j$ and the law of the $a_i${}$)$, a small enough constant $\epsilon > 0$, a constant $B > 0$, choices of even $m = m(n) \sim C_j \rho_n \log n$, and a polynomial $B(x) = \sum_{i = 0}^{m - 1} b_i x^i$, with $\lvert b_i \rvert \le B$ and $b_i \in \mathcal{V}$, such that the following holds. Suppose $g(x) = \sum_{i = 0}^{n - 1} g_i x^i$ is any polynomial satisfying the following four properties: 
\smallskip
\\
\indent \emph{\textbf{H1}:} \qquad $\lvert g_i - b_i \rvert < \epsilon$ for $i \in \{0, 1, \dots , m - 1\}$\\
\indent \emph{\textbf{H2}:} \qquad $\lvert g_{n - 1 - i}\rvert \le M$ for $i \in \{0, 1, \dots , m - 1\}$\\
\indent \emph{\textbf{H3}:} \qquad $g_m + g_{m + 1} x + \cdots + g_{n - m - 1} x^{n - 2m - 1} > n^{-1/4} \sigma_{n - 2m}(x)$ for all $x \in \R$\\
\indent \emph{\textbf{H4}:} \qquad $\lvert g_i \rvert < (n-1)^{\rho_{n-1}}$ for $i \in \{0, 1, \dots , n - 1\}$.\\
\smallskip
Then $g(x)$ has exactly $j$ zeros in $[0, 1]$, all of them simple, and is positive on $[-1, 0]$.
\end{lemma}
This is proved as \textit{Step 1} and \textit{Step 2} of Lemma 8.2 in DPSZ. Our property \textbf{H2} replaces the stronger property \textbf{A2} in DPSZ, but one sees readily in their proofs of \textit{Step 1} and \textit{Step 2} that in each of the two occurrences where \textbf{A2} is invoked, only the estimate 
\begin{equation} \label{eq-whatsused}
\lvert g_{n - 1} x^{n - 1} + ... + g_{n - m} x^{n - m} \rvert \le m M \lvert x^{n - m} \rvert
\end{equation}
on $\lvert x \rvert \le 1$ is necessary. This estimate follows immediately from \textbf{H2}. 
We now show that with large enough probability, $f_{n, S}$ satisfies simultaneously \textbf{H1}, \textbf{H2}, \textbf{H3}, and \textbf{H4}, which we may call satisfying \textbf{H}. The proof is essentially the same as the proof of Lemma 8.1 in DPSZ, but it is nonetheless short and instructive.  
\begin{lemma} \label{lem-Hprob}
Take $M > \max_{i \in S} |c_i|$. The probability that $f_{n, S}$ satisfies $\mathbf{H}$ is at least $n^{-b + o(1)}$. 
\end{lemma}
\begin{proof}
It is easy to see that, because $j = o(\log n)$, we have $\P(\mathbf{H1}) \geq n^{o(1)}$ and $\P(\mathbf{H2}) \geq n^{o(1)}$. By~\cite[Theorem~1.3]{MR1915821} (i.e., the analogue of our Theorem~\ref{thm-locmaingen} for fully random polynomials), we have that $\P(\textbf{H3}) \geq n^{-b + o(1)}$. Because \textbf{H1}, \textbf{H2}, and \textbf{H3} are independent, the probability that they occur together is at least $n^{-b + o(1)}$. As for \textbf{H4} note that by Markov's inequality, we have for each $i$ that $$\P(|a_i| \ge (n-1)^{\rho_{n-1}}) = \P( |a_i|^{p_{n-1}} \ge n^{p_{n-1}\rho_{n-1}}) \ll n^{-5},$$ so \textbf{H4} fails with probability $O(n^{-3})$. Using the fact that $\P(\mathbf{H}) \geq \P(\mathbf{H1}, \mathbf{H2}, \mathbf{H3}) - \P(\text{not }\mathbf{H4})$, together with the fact that $b < 3$, the desired lower bound follows.
\end{proof}

Now, let
$\mathbf{H'}$ be the condition on $g(x) = \sum_{i = 0}^{n - 1} g_i x^i$ that $g_{n - i} = c_i$ for each $i \in S$, that
\textbf{H} also holds, and further that
\begin{equation} \label{eq-smallpolybigforxbig}
x^m \frac{\sigma_{n - 2m}(x)}{n^{1/4}} - \sum_{i = 0}^{m - 1} (B + \epsilon) \lvert x \rvert^i + \sum_{i \in S} c_i x^{n - i} + 
\sum_{\substack{i \in \{1, \dots , m\} \\ i \not\in S}} g_{n - i} x^{n - i}
 > 0, \text{ }
\forall x \text{ such that }\lvert x \rvert \ge 1.
\end{equation}
Observe that the condition $\mathbf{H'}$ may be obtained from the condition $\mathbf{H}$ by keeping $\mathbf{H1}$, $\mathbf{H3}$, and $\mathbf{H4}$ intact and by replacing $\mathbf{H2}$ with the stronger condition that $|g_{n-1-i}| \leq M$ for $i \in \{0,\dots,m-1\}$ and~\eqref{eq-smallpolybigforxbig} holds.
\begin{lemma}
Suppose $g(x) = \sum_{i = 0}^{n - 1} g_i x^i$ satisfies $\mathbf{H'}$. Then $g(x)$ has exactly $j$ real zeros, all of them simple. 
\end{lemma}
\begin{proof}
It only remains to be shown that if $\mathbf{H'}$ holds, then $g(x)$ has no real zeros with $\lvert x \rvert \ge 1$, which holds by combining \eqref{eq-smallpolybigforxbig}, \textbf{H1}, and \textbf{H3}.
\end{proof}
\begin{lemma}
The probability that $f_{n, S}$ satisfies $\mathbf{H'}$ is at least $n^{-b + o(1)}$.
\end{lemma}
\begin{proof}
By the proof of Lemma \ref{lem-Hprob}, it suffices to show that $f_{n, S}$ satisfies \eqref{eq-smallpolybigforxbig} with probability at least $n^{o(1)}$. For $\lvert x \rvert \ge 1$, we have $\sigma_{n - 2m}(x) \ge (n - 2m)^{1/2}$. Consequently, for all $n \gg 1$, we have
\[
x^m n^{-1/4} \sigma_{n - 2m}(x) - \sum_{i = 0}^{m - 1} (B + \epsilon) \lvert x \rvert^i > 0
\]
for all $x$ such that $\lvert x \rvert \ge 1$. It now suffices to prove that 
\[
\sum_{i \in S} c_i x^{m - i} + 
\sum_{\substack{i \in \{1, ... , m\} \\ i \not\in S}} a_{m - i} x^{m - i}\neq 0
\]
for $x$ such that $|x| > 1$ with probability at least $n^{o(1)}$, but this follows from the niceness condition, just as in the proof of Lemma~\ref{lower}.
\end{proof}
 In the context of Theorem~\ref{thm-fixedcoeffszero2}, it is also clear that if $k = k(n)$ grows as slowly as $o(\log n)$ that the above estimates continue to hold --- Lemma~\ref{lem-hconds} needs to be modified to permit fixing more than $m$ coefficients at the top of the polynomial, but this is possible because the behavior of the polynomial on $[-1,1]$ is determined almost entirely by the last $m$ coefficients.

\subsubsection{The case where $n-1,j$ are odd} \label{sec-lowkeven1}

Let $n-1,j$ be odd. Recall that in \S\ref{sec-lowkeven1}, our idea was to adapt the proof given in~\cite[\S8.1]{MR1915821} to work in our setting. This was possible because the proof in \emph{loc.~cit.}~involves choosing very tiny ranges for the low-degree coefficients and only requires that the high-degree coefficients be bounded. But when $n-1,j$ are odd, the proof given in~\cite[\S8.2]{MR1915821} does not similarly adapt to our setting: their argument involves choosing tiny ranges for \emph{both} the low- \emph{and} high-degree coefficients of the polynomial, which prevents us from fixing the high-degree coefficients to have the desired values $c_i$. Thus, to handle the case where $n-1,j$ are odd, we must prove a new analogue of Lemma~\ref{lem-hconds} that allows us to control the number of zeros using only the low-degree coefficients.

Since the $a_i$ are of zero mean and unit variance, there exist $\alpha < 0 < \beta$ such that $\alpha,\beta \in \mc{V}$. We may assume without loss of generality that $-\alpha > \beta$. Let $s \geq 4$ be an even integer such that $\alpha + (s-1)\beta < 0$, and define polynomials
$$Q(x) \defeq -\beta x^{s-1} - \sum_{i = 0}^{s-2} \alpha x^i,\qquad R(x) = -\alpha - \sum_{i = 1}^{s-1} \beta x^i.$$
Note that by construction we have $Q(x) < 0$ for all $x$ such that $|x| \leq 1$ and $R(1) > 0 > R(-1)$. 

Let $\delta > 0$ be sufficiently small, let $r = r(\delta)$ be sufficiently large, and let $\epsilon = \epsilon(r,\delta)$ be sufficiently small; we will choose the values of these quantities, all of which are independent of $n$, later. Let $r_i$ denote the multiple of $s$ nearest to $r^i$, for each $i \in \{1,\dots, j\}$. Let $m = m(n)$ be the multiple of $s$ nearest to $2r_k\rho_n \log n/|\log(1-\delta)|$. Define the polynomial
\begin{align*}
    & B(x) = \sum_{i = 0}^{m} b_ix^i \defeq (1 + x^s + x^{2s} + \cdots + x^{r_1-s})Q(x) + (x^{r_1} + x^{r_1 + s} + \cdots + x^{r_2 -s})R(x) + \\[-9pt]
    &\qquad\qquad\qquad\qquad\qquad (x^{r_2} + x^{r_2+s}  + \cdots + x^{r_3-s})Q(x) + (x^{r_3} + x^{r_3 + s} + \cdots + x^{r_4 -s})R(x) + \cdots + \\
    &\qquad\qquad\qquad\qquad\qquad\qquad (x^{r_{j-1}} + x^{r_{j-1}+s} + \cdots + x^{r_j-s})Q(x) + (x^{r_j} + x^{r_j + s} + \cdots + x^{m -s})R(x) + \alpha x^m. 
\end{align*}
Then we have the following lemma (which also applies in the setting of~\cite[\S8.2]{MR1915821} and gives a different proof of their result for fully random polynomials):
\begin{lemma} \label{lem-hconds2}
For suitable choices of $\delta$, $r$, and $\epsilon$, suppose that $g(x) = \sum_{i = 0}^{n - 1} g_i x^i$ is any polynomial satisfying the following four properties: 
\smallskip
\\
\indent \emph{\textbf{K1}:} \qquad $\lvert g_i - b_i \rvert < \epsilon$ for $i \in \{0, 1, \dots , m\}$\\
\indent \emph{\textbf{K2}:} \qquad $\lvert g_{n - 1 - i}\rvert \le M$ for $i \in \{0, 1, \dots , m - 1\}$\\
\indent \emph{\textbf{K3}:} \qquad $g_{m+1} + g_{m + 2} x + \cdots + g_{n - m - 1} x^{n - 2m - 2} > n^{-1/4} \sigma_{n - 2m-2}(x)$ for all $x \in \R$\\
\indent \emph{\textbf{K4}:} \qquad $\lvert g_i \rvert < (n-1)^{\rho_{n-1}}$ for $i \in \{0, 1, \dots , n - 1\}$.\\
\medskip
Then $g(x)$ has exactly $j$ zeros in $[0, 1]$, all of them simple, and is negative on $[-1, 0]$.
\end{lemma}
\begin{proof}
    The proof of the lemma follows the argument given in~\cite[\S8.1]{MR1915821} for the proof of Lemma~\ref{lem-hconds}, and it is likewise divided into three steps. The first step is to prove that $g$ has the desired behavior on $[0,1]$; the second step is to do the same for $[-1,0]$; and the third step is to prove the same for $\R \smallsetminus [-1,1]$. The third step is entirely identical to that of the proof of Lemma~\ref{lem-hconds}, so we omit it and focus on the first two steps.

    \smallskip \noindent{\emph{Step 1}}. The zeros of $g(x)$ in $(0,1)$ are the same as those of $F(x) \defeq (1 - x^s)g(x)$, so it suffices to prove the following conditions, which ensure that $g(x)$ has exactly $j$ zeros on $[0,1]$:
    \begin{itemize}
        \item $F(x) < 0$ for $x \in [0,\delta^{1/r_1}]$
        \item $F'(x) > 0$ for $x \in [\delta^{1/r_1},(1-\delta)^{1/r_1}]$
        \item $F(x) > 0$ for $x \in [(1-\delta)^{1/r_1},\delta^{1/r_2}]$
        \item $F'(x) < 0$ for $x \in [\delta^{1/r_2},(1-\delta)^{1/r_2}]$
        \item $F(x) < 0$ for $x \in [(1-\delta)^{1/r_2},\delta^{1/r_3}]$
        \item $F'(x) > 0$ for $x \in [\delta^{1/r_3},(1-\delta)^{1/r_3}]$\\
        \hspace*{75pt}$\vdots$
        \item $F'(x) > 0$ for $x \in [\delta^{1/r_j},(1-\delta)^{1/r_j}]$
        \item $F(x) > 0$ for $x \in [(1-\delta)^{1/r_j},2^{-1/m}]$
        \item $g(x) > 0$ for $x \in [2^{-1/m},1]$
    \end{itemize}
    To proceed, first note that for all $x$ such that $|x| \leq (1-\delta)^{1/r_k}$, the polynomial $F$ is approximated very well by the polynomial $(1-x^s)B(x)$. Indeed, by conditions $\mathbf{K1}$ and $\mathbf{K4}$, we have the following bound for all $x$ such that $|x| \leq (1-\delta)^{1/r_k}$ and all $n \gg 1$:
    \begin{align} \label{eq-wellapprox1}
        \big|F(x) - (1-x^s)B(x)\big| & \leq (1-x^s) \bigg(\epsilon(1 + |x| + \cdots + |x|^{m-1}) + n^{\rho_n}(|x|^{m+1} + \cdots + |x|^n\bigg) \\
        & \leq (1-|x|)^{-1}(\epsilon + n^{\rho_n}x^{m+1}) \ll \epsilon, \nonumber
    \end{align}
    where the implied constant depends on $r$ and $\delta$. We are thus led to consider the polynomial $(1-x^s)B(x)$, which can be expanded as follows:
    \begin{equation} \label{eq-expandB}
        (1-x^s)B(x) = Q(x) + \bigg(\sum_{\ell = 1}^j(-1)^\ell\big(Q(x) - R(x)\big)x^{r_\ell}\bigg) - R(x)x^m + \alpha x^m(1-x^s).
    \end{equation}
    We start by verifying that the signs of $F(x)$ are correct on the relevant intervals. Fix $N$ such that $\sup_{x \in [0,1]} \max\{|Q(x)|,|R(x)|\} \leq N$, and take $x \in [0,\delta^{1/r_1}]$. Then the factors $x^{r_\ell}$, $x^m$, and $x^{m-s}$ occurring in~\eqref{eq-expandB} are all at most $\delta$ in size, so we have
    \begin{equation} \label{eq-boundbbelow1}
        (1-x^s)B(x) \leq Q(x) + (2j+3)N\delta.
    \end{equation}
Thus, for all sufficiently small $\delta$, the negativity of $Q(x)$ on $[0,1]$ implies in conjunction with~\eqref{eq-boundbbelow1} that $(1-x^s)B(x)$ is negative and bounded away from zero for all $x \in [0,\delta^{1/r_1}]$. Using~\eqref{eq-wellapprox1}, for all sufficiently small $\epsilon$ we conclude that $F(x)$ is negative on this interval.

Next, take $x \in [(1-\delta)^{1/r_i},\delta^{1/r_{i+1}}]$ for some $i \in \{1,\dots, j-1\}$, and note that we have $x^m \leq x^{r_\ell} \leq \delta$ for all $\ell > i$ and $x^{r_\ell} \in [1-\delta,1]$ for all $\ell \leq i$. Observe that we have the identity
\begin{equation} \label{eq-qrindicator}
    Q(x) + \sum_{\ell = 1}^i(-1)^\ell\big(Q(x) - R(x)\big) = QR(x) \defeq \begin{cases} Q(x), & \text{if $i$ is even,} \\ R(x), & \text{if $i$ is odd.}
    \end{cases}
\end{equation}
Combining~\eqref{eq-qrindicator} with~\eqref{eq-expandB}, it follows that for all $x \in [(1-\delta)^{1/r_i},\delta^{1/r_{i+1}}]$  we have
\begin{equation} \label{eq-bcompqr}
    \big|(1-x^s)B(x) - QR(x)\big| \leq (2j+3)N\delta.
\end{equation}
Thus, for all sufficiently small $\delta$, the polynomial $(1-x^s)B(x)$ is as close as we like to $QR(x)$, which is negative if $i$ is even, and which is positive if $i$ is odd and $r$ is sufficiently large. That we have the desired behavior on this interval of $x$ then follows from~\eqref{eq-wellapprox1} by taking $\epsilon$ to be sufficiently small.

Next, take $x \in [(1-\delta)^{1/r_j},2^{-1/m}]$ --- note that on this interval, the bound~\eqref{eq-wellapprox1} no longer applies, but it follows from conditions $\mathbf{K1}$, $\mathbf{K2}$, and $\mathbf{K3}$ that for all $n \gg 1$,
\begin{equation} \label{eq-applyK3}
    F(x) - (1-x^s)B(x) \geq -(1-x^s)\big(\epsilon(1 + |x| + \cdots + |x|^{m}) + mM|x|^{n-m}\big) \gg -\epsilon,
\end{equation}
where the implied constant is positive and depends on $r$ and $\delta$.
Then, from~\eqref{eq-expandB} and~\eqref{eq-bcompqr}, taking $i = j$, we deduce that
\begin{equation} \label{eq-thisthetightbound}
    (1-x^s)B(x) \geq \frac{1}{2}R(x) - 2jN\delta.
\end{equation}
Thus, for all sufficiently small $\delta$, the polynomial $(1-x^s)B(x)$ is bounded away from zero, and the desired behavior on this interval of $x$ then follows from~\eqref{eq-applyK3} by taking $\epsilon$ to be sufficiently small.

    Next, take $x \in [2^{-1/m},\delta^{1/n}]$. Applying conditions $\mathbf{K1}$, $\mathbf{K2}$, and $\mathbf{K3}$ to control the three ranges of terms in the polynomial $g$ yields the following lower bound:
    \begin{align}
        g(x) & \geq B(x) - \epsilon (1 + \cdots + |x|^m) - mMx^{n-m} \geq (x^{r_j} + x^{r_j + s} + \cdots + x^{m-s})R(x) - r_jN - \epsilon m - mM\delta^{1/2} \nonumber \\
        & \geq \left(\frac{m-r_j}{2s}\right)R(x) - r_jN - \epsilon m - mM\delta^{1/2}\label{eq-lmhbound}
    \end{align}
    Since $m = m(n) \to \infty$ as $n \to \infty$, and since $R(1) > 0$, we see the bound in~\eqref{eq-lmhbound} is positive for all $n \gg 1$ so long as $\delta$ and $\epsilon$ are sufficiently small.

    Next, take $x \in [\delta^{1/n},1]$. In this range, conditions $\mathbf{K1}$, $\mathbf{K2}$, and $\mathbf{K3}$ imply that the terms $g_{m+1}x^{m+1} + \cdots + g_{n-m-1}x^{n-m-1}$ dominate, contributing at least $n^{1/8}$ as $n \to \infty$. The other terms contribute $\ll m = o(\log n)$ on this range, giving the desired positivity.

    \smallskip

    We now move on to verifying that the signs of $F'(x)$ are correct on the relevant intervals. We note that for all $x \in [0,(1-\delta)^{1/r_j}]$, the derivative $F'(x)$ is approximated very well by the derivative of $(1-x^s)B(x)$. Indeed, by conditions $\mathbf{K1}$ and $\mathbf{K4}$, we have the following bound for all $x \in [0,(1-\delta)^{1/r_j}]$ and all $n \gg 1$:
    \begin{align} \label{eq-wellapprox2}
        \left|F'(x) - \frac{d}{dx}\big((1-x^s)B(x)\big)\right| & \leq sx^{s-1}\big(\epsilon(1 + \cdots + x^{m}) + n^{\rho_n}(x^{m+1} + \cdots + x^n)\big) + \\[-7pt]
        & \qquad\quad (1-x^s)\big(\epsilon(1 + 2x + \cdots + mx^{m-1}) + n^{\rho_n}((m+1)x^m + \cdots + nx^{n-1})\big) \nonumber \\
        & \leq s(1-x)^{-1}\big(\epsilon + n^{\rho_n}x^{m+1}\big) + (1-x)^{-2}(\epsilon + n^{\rho_n}(m+1)x^m) \ll \epsilon, \nonumber
    \end{align}
    where the implied constant depends on $r$ and $\delta$. By differentiating~\eqref{eq-expandB}, we obtain the identity
    \begin{equation} \label{eq-expandB'}
        \frac{d}{dx}\big((1-x^s)B(x)\big) = Q'(x) + \sum_{\ell = 1}^j (-1)^{\ell}\left(\big(Q'(x) - R'(x)\big)x^{r_\ell} + \big(Q(x)-R(x)\big)r_\ell x^{r_\ell - 1}\right) - o(1).
    \end{equation}
    Even though our choice of the polynomial $B$ is different from the choice made in~\cite[\S8]{MR1915821}, and we chose $Q$ and $R$ to be the negatives of the choices made in \emph{loc.~cit.}, the identity~\eqref{eq-expandB} takes the exact same shape in both settings, and thus their analysis applies with minimal change. In particular, their argument shows that the term $(-1)^i\big(Q(x) - R(x)\big)r_ix^{r_i-1}$ dominates the right-hand side of~\eqref{eq-expandB'} on the interval $[\delta^{1/r_i},(1-\delta)^{1/r_i}]$, and the desired behavior then follows from the fact that $Q(x) - R(x)$ is negative and bounded away from zero on the interval $[\delta^{1/r_1},1]$, so long as $r$ is sufficiently large. 
    
    \medskip \noindent{\emph{Step 2}}. The analysis here is similar to that of \emph{Step 1}, except that both $Q(x)$ and $R(x)$ are negative when $x$ is close to $-1$, as is the middle range of terms $g_{m+1}x^{m+1} + \cdots + g_{n-m-1}x^{n-2m-2}$. For the sake of brevity, we just highlight two important places in which the proof in \emph{Step 1} needs to be modified here. Firstly, in~\eqref{eq-thisthetightbound}, the contribution of $\alpha x^m(1-x^s)$ cannot be ignored, but we observe that this term is $\ll \delta$ on the interval $[-2^{-1/m},-(1-\delta)^{1/r_j}]$, so it is negligible as long as $\delta$ is sufficiently small. Secondly, we are no longer interested in studying the sign of the derivative $F'(x)$; rather, we must show that $F(x) < 0$ on the intervals $[-(1-\delta)^{1/r_i},-\delta^{1/r_i}]$. To do this, it suffices by~\eqref{eq-wellapprox1} (taking $\epsilon$ to be sufficiently small) to prove that $(1-x^s)B(x)$ is negative and bounded away from zero on $[-(1-\delta)^{1/r_i},-\delta^{1/r_i}]$. Combining~\eqref{eq-expandB} and~\eqref{eq-qrindicator}, we see that for all $x \in [-(1-\delta)^{1/r_i},-\delta^{1/r_i}]$,
    \begin{equation} \label{eq-negativebounds}
        \big|(1-x^s)B(x) - t_i(x)Q(x) - (1-t_i(x))R(x)\big| \leq (2j+1)N\delta,
    \end{equation}
    where $t_i(x) = 1 - x^{r_i}$ if $i$ is even and $t_i(x) = x^{r_i}$ if $i$ is odd. Thus, if $\delta$ is sufficiently small, it suffices to see that $t_i(x)Q(x) + (1-t_i(x))R(x)$ is negative and bounded away from zero, but this is true for all sufficiently large $r$ because then both $Q(x)$ and $R(x)$ are negative and bounded away from zero.
\end{proof}

The lower bound in Theorem~\ref{thm-locmaingen} in the case where $n-1,j$ are odd now follows from Lemma~\ref{lem-hconds2} in exactly the same way as the case where $n-1,j$ are even follows from Lemma~\ref{lem-hconds}. In the context of Theorem~\ref{thm-fixedcoeffszero2}, it is also clear that if $k = k(n)$ grows as slowly as $o(\log n)$ that the above estimates continue to hold --- Lemma~\ref{lem-hconds2} needs to be modified to permit fixing more than $m$ coefficients at the top of the polynomial, but this is possible because the behavior of the polynomial on $[-1,1]$ is determined almost entirely by $B(x)$.

\subsection{Upper bound}

\noindent Fix $\delta \in (0, \frac{1}{2})$, and let $V$ be as defined in~\eqref{eq-defofv}. Then restricting $x$ to lie in $V$, we have that 
\begin{equation}
P_{n, j} \le \P\big(\# \{x \in \R: f_{n,S}(x) = 0\} \le j\big) \le 
\P\big(\# \{x \in V: f_{n,S}(x) = 0\} \le j\big).
\end{equation}
Consider the range $T = [\delta \log n, (1 - \delta) \log n]$, chosen so that setting $x = w_1(t) = 1 - e^{-t}$ for $t \in T$ traces out the interval $\mc{I}_1 \defeq [1 - n^{-\delta}, 1 - n^{-(1 - \delta)}]$, while $x = w_2(t) = w_1(t)^{-1}$, $x = w_3(t) = -w_1(t)^{-1}$, and $x = w_4(t) = -w_1(t)$ trace out the remaining closed intervals $\mc{I}_2$, $\mc{I}_3$, and $\mc{I}_4$ that constitute $V$.
Divide $T$ into $R \defeq \lfloor (1 - 2 \delta) \log n \rfloor$ unit-length intervals --- the last one slightly longer if necessary --- with images $J_{(i - 1) R + 1}, ... , J_{iR}$ in $\mc{I}_i$ under $w_i$. If $f_{n,S}$ indeed has at most $j$ zeros in $V$, then there is at least one way of choosing $j$ intervals $J_{i_1}, ... , J_{i_j}$ to ignore, so that $f_{n,S}$ has constant sign on each of the $\ell \le j + 4$ maximal leftover intervals comprising $V \smallsetminus \big(J_{i_1} \cup \cdots \cup J_{i_j}\big)$. Number these maximal leftover intervals $L_1, ... , L_{\ell}$. A quick calculation reveals that, because $j = o\big(\frac{\log n}{\log \log n}\big)$, there are $n^{o(1)}$ choices of $(i_1, ... , i_j)$ and $n^{o(1)}$ choices of $(s_i)_i \in \{\pm 1\}^\ell$ for the signs of $f_{n,S}$ on the $L_i$. Thus, it suffices to obtain an upper bound on the quantity
\begin{equation} \label{eq-thingtoest}
\P\big(\min_{1 \leq i \leq \ell} \inf_{x \in L_i} s_i f_{n,S}(x) >  0 \big)
\end{equation}
for each choice of intervals $J_{i_1},\dots,J_{i_j}$ to ignore and each choice of signs $(s_i)_i \in \{\pm 1\}^{\ell}$.

To estimate~\eqref{eq-thingtoest}, we split off the first $\approx k$ terms of $f_{n,S}$, leaving behind a fully random polynomial to which we can apply the results of~\cite{MR1915821}. More precisely, let $r$ be the smallest even integer such that $r \geq k$. 
Then we write
\begin{equation} \label{eq-splitoffs}
f_{n,S}(x) = \sigma_{n - r}(x) \hat{f}_{n - r}(x) + \sum_{i \in S} 
c_i x^{n - i} + \sum_{\substack{i \in \{1, \dots , r\} \\ i \not\in S}} a_{n-i} x^{n - i}.
\end{equation}
For $s_if_{n,S}(x)$ to be positive, $s_i\sigma_{n-r}\hat{f}_{n-r}(x)$ must be at least as large as the negative of the sum of the remaining terms on the right-hand side~\eqref{eq-splitoffs}. To bound these remaining terms from above, observe that there exists a constant $\xi > 0$ such that for all $x \in V$, $i \in \{1, \dots, r\}$, and $n \gg 1$, we have 
\begin{equation} \label{eq-powerbound}
|x^{n-i}| \leq \xi \times n^{-\delta/2} \sigma_{n - r}(x).
\end{equation} Applying the bound in~\eqref{eq-powerbound} to the terms of degree $n-i$ where $i \in S$, we find for all $n \gg 1$ that
\begin{align} \label{eq-seprand}
\P\bigg(\min_{1 \leq i \leq \ell} \inf_{x \in L_i} s_i f_{n,S}(x) >  0 \bigg) \le 
\P\bigg(\min_{1 \leq i \leq \ell} \inf_{x \in L_i} s_i \hat{f}_{n - r}(x) > - r \xi n^{-\delta/4} \bigg)
+
\sum_{\substack{i \in \{1, \dots , r\} \\ i \not\in S}} \P\big(|a_i| \ge n^{\delta/4} \big).
\end{align}
Note that the summation on the right-hand side of~\eqref{eq-seprand} is $O(n^{-3})$ --- indeed, it is $O(n^{-a})$ for every $a > 0$, because the coefficient law has finite moments of all orders. As for the first term on the right-hand side of~\eqref{eq-seprand}, this is shown in~\cite{MR1915821} to be bounded by $n^{-(1-2\delta)b + o(1)}$. Taking the limit as $\delta \to 0$ yields the desired upper bound. It is also clear that if $k = k(n)$ (and hence $r$) grows as slowly as $o(\log n)$ that the above estimates continue to hold.

\section*{Acknowledgments}

 \noindent  
 We are grateful to Kumar Murty, Bjorn Poonen, Qi-Man Shao, and Melanie Matchett Wood for helpful conversations. We also thank Will FitzGerald for pointing us to the references~\cite{MR4499280} and~\cite{PhysRevLett.121.150601}. AS was supported by the National Science Foundation, under the Graduate Research Fellowship, as well as Award No.~2202839.

\bibliographystyle{alpha}
\bibliography{bibfile}

\begin{thebibliography}{BCFG22}

\bibitem[Adl90]{adler1990introduction}
{R}.~J. Adler.
\newblock {\em An introduction to continuity, extrema, and related topics for
  general {G}aussian processes}, volume~12 of {\em Institute of Mathematical
  Statistics Lecture Notes---Monograph Series}.
\newblock Institute of Mathematical Statistics, Hayward, CA, 1990.

\bibitem[AP14a]{MR3238322}
{S}. Akiyama and {A}. Peth\H{o}.
\newblock On the distribution of polynomials with bounded roots, {I}.
  {P}olynomials with real coefficients.
\newblock {\em J. Math. Soc. Japan}, 66(3):927--949, 2014.

\bibitem[AP14b]{MR3237072}
{S}. Akiyama and {A}. Peth\H{o}.
\newblock On the distribution of polynomials with bounded roots {II}.
  {P}olynomials with integer coefficients.
\newblock {\em Unif. Distrib. Theory}, 9(1):5--19, 2014.

\bibitem[BCFG22]{MR4422614}
{M}. Bhargava, {J}. Cremona, {T}. Fisher, and {S}. Gajovi\'{c}.
\newblock The density of polynomials of degree {$n$} over {$\Bbb Z_p$} having
  exactly {$r$} roots in {$\Bbb Q_p$}.
\newblock {\em Proc. Lond. Math. Soc. (3)}, 124(5):713--736, 2022.

\bibitem[BGW17]{MR3600041}
M.~Bhargava, B.~H. Gross, and X.~Wang.
\newblock A positive proportion of locally soluble hyperelliptic curves over
  {$\Bbb Q$} have no point over any odd degree extension.
\newblock {\em J. Amer. Math. Soc.}, 30(2):451--493, 2017.
\newblock With an appendix by {T}.~Dokchitser and {V}.~Dokchitser.

\bibitem[Bha13]{thesource}
{M}. Bhargava.
\newblock Most hyperelliptic curves over $\mathbb{Q}$ have no rational points.
\newblock {\em arXiv preprint arXiv:1308.0395}, 2013.

\bibitem[BP32]{blochpolya}
A.~Bloch and G.~P\'olya.
\newblock On the roots of certain algebraic equations.
\newblock {\em Proc. London Math. Soc.}, 33:102--114, 1932.

\bibitem[CH17]{MR3687945}
{F}. Calegari and {Z}. Huang.
\newblock Counting {P}erron numbers by absolute value.
\newblock {\em J. Lond. Math. Soc. (2)}, 96(1):181--200, 2017.

\bibitem[DM15]{MR3298469}
{A}. Dembo and {S}. Mukherjee.
\newblock No zero-crossings for random polynomials and the heat equation.
\newblock {\em Ann. Probab.}, 43(1):85--118, 2015.

\bibitem[DPSZ02]{MR1915821}
{A}. Dembo, {B}. Poonen, {Q}.-{M}. Shao, and {O}. Zeitouni.
\newblock Random polynomials having few or no real zeros.
\newblock {\em J. Amer. Math. Soc.}, 15(4):857--892, 2002.

\bibitem[FTZ22]{MR4499280}
{W}. FitzGerald, {R}. Tribe, and {O}. Zaboronski.
\newblock Asymptotic expansions for a class of {F}redholm {P}faffians and
  interacting particle systems.
\newblock {\em Ann. Probab.}, 50(6):2409--2474, 2022.

\bibitem[GWY23]{g2023chebotarev}
{A}. G, {Y}. Wei, and {J}. Yin.
\newblock A {C}hebotarev density theorem over local fields.
\newblock {\em arXiv preprint arXiv:2212.00294}, 2023.

\bibitem[IM68]{IM1}
I.~A. Ibragimov and N.~B. Maslova.
\newblock The average number of zeros of random polynomials.
\newblock {\em Vestnik Leningrad. Univ.}, 23:171--172, 1968.

\bibitem[IM71a]{IM3}
I.~A. Ibragimov and N.~B. Maslova.
\newblock The average number of real roots of random polynomials.
\newblock {\em Soviet Math. Dokl.}, 12:1004--1008, 1971.

\bibitem[IM71b]{IM2}
I.~A. Ibragimov and N.~B. Maslova.
\newblock The mean number of real zeros of random polynomials. i. coefficients
  with zero mean.
\newblock {\em Theor. Probability Appl.}, 16:228--248, 1971.

\bibitem[Kac43]{kac1}
M.~Kac.
\newblock On the average number of real roots of a random algebraic equation.
\newblock {\em Bull. Amer. Math. Soc.}, 49:314--320, 1943. Erratum:
  \textit{Bull. Amer. Math. Soc.}, 49:938, 1943.

\bibitem[Kac49]{kac2}
M.~Kac.
\newblock On the average number of real roots of a random algebraic equation.
  ii.
\newblock {\em Proc. London Math. Soc.}, 50:390--408, 1949.

\bibitem[KMT76]{MR0402883}
J.~Koml\'{o}s, P.~Major, and G.~Tusn\'{a}dy.
\newblock An approximation of partial sums of independent {RV}'s, and the
  sample {DF}. {II}.
\newblock {\em Z. Wahrscheinlichkeitstheorie und Verw. Gebiete}, 34(1):33--58,
  1976.

\bibitem[LO38]{LO1}
J.~E. Littlewood and A.~C. Offord.
\newblock On the number of real roots of a random algebraic equation. i.
\newblock {\em J. London Math. Soc.}, 13:288--295, 1938.

\bibitem[LO39]{LO2}
J.~E. Littlewood and A.~C. Offord.
\newblock On the number of real roots of a random algebraic equation. ii.
\newblock {\em Proc. Cambridge Philos. Soc.}, 35:133--148, 1939.

\bibitem[LO43]{LO3}
J.~E. Littlewood and A.~C. Offord.
\newblock On the number of real roots of a random algebraic equation. iii.
\newblock {\em Rec. Math. [Mat. Sbornik] N.S.}, 54:277--286, 1943.

\bibitem[LS02]{MR1902188}
{W}.~{V}. Li and {Q}.-{M}. Shao.
\newblock A normal comparison inequality and its applications.
\newblock {\em Probab. Theory Related Fields}, 122(4):494--508, 2002.

\bibitem[Mas74]{Maslova}
N.~B. Maslova.
\newblock The distribution of the number of real roots of random polynomials.
\newblock {\em Theor. Probability Appl.}, 19:461--473, 1974.

\bibitem[Mol12]{MR2999109}
G.~Molchan.
\newblock Survival exponents for some {G}aussian processes.
\newblock {\em Int. J. Stoch. Anal.}, pages Art. ID 137271, 20, 2012.

\bibitem[PS99]{poonenstoll}
B.~Poonen and M.~Stoll.
\newblock The cassels–tate pairing on polarized abelian varieties.
\newblock {\em Annals of Math.}, 150:1109--1149, 1999.

\bibitem[PS18]{PhysRevLett.121.150601}
{M}. Poplavskyi and {G}. Schehr.
\newblock Exact persistence exponent for the $2{D}$-diffusion equation and
  related {K}ac polynomials.
\newblock {\em Phys. Rev. Lett.}, 121:150601, Oct 2018.

\bibitem[SM07]{PhysRevLett.99.060603}
{G}. Schehr and {S}.~{N}. Majumdar.
\newblock Statistics of the number of zero crossings: From random polynomials
  to the diffusion equation.
\newblock {\em Phys. Rev. Lett.}, 99:060603, Aug 2007.

\bibitem[SM08]{MR2415102}
{G}. Schehr and {S}.~{N}. Majumdar.
\newblock Real roots of random polynomials and zero crossing properties of
  diffusion equation.
\newblock {\em J. Stat. Phys.}, 132(2):235--273, 2008.

\bibitem[Sto74]{MR0455094}
{W}.~{F}. Stout.
\newblock {\em Almost sure convergence}, volume Vol. 24 of {\em Probability and
  Mathematical Statistics}.
\newblock Academic Press [Harcourt Brace Jovanovich, Publishers], New
  York-London, 1974.

\bibitem[Swa24]{MR4703126}
{A}. Swaminathan.
\newblock Most odd-degree binary forms fail to primitively represent a square.
\newblock {\em Compos. Math.}, 160(3):481--517, 2024.

\end{thebibliography}

\end{document}